\documentclass[a4paper,10pt]{article}

\usepackage{amsmath}                    % From AMS/CTAN distribution
\usepackage{amssymb}                    % From AMS/CTAN distribution
\usepackage{amsthm}                     % From AMS/CTAN distribution
\usepackage{amsfonts}                   % From AMS/CTAN distrib
\usepackage{subfigure}
\usepackage[dvips]{graphicx}
\usepackage{comment}

\newtheorem{theo}{Theorem}[section]

\newtheorem{pro}[theo]{Proposition}
\newtheorem{lem}[theo]{Lemma}
\newtheorem{rem}[theo]{Remark}
\newtheorem{ass}[theo]{Assumption}
\newtheorem{proposition}[theo]{Proposition}

\newtheorem{cor}[theo]{Corollary}

\newcommand{\trace}{{\mathrm{trace}}}
\newcommand{\loc}{{\mathrm{loc}}}

\newcommand{\hm}{\widehat{m}}
\renewcommand{\hat}{\widehat}

\newcommand{\cf}{\mathcal{F}}
\newcommand{\EX}{\mathbb{E}}
\newcommand{\bbN}{\mathbb{N}}
\newcommand{\bbR}{\mathbb{R}}
\newcommand{\cB}{\mathcal{B}}

\newcommand{\cN}{\mathcal{N}}
\newcommand{\WP}{\mathcal{W}}

\def \R{\mathbb R}

\def \IP{\mathbb P}

\newcommand{\Z}{\mathbb{Z}}

\newcommand{\cH}{\mathcal{H}}
\newcommand{\cV}{\mathcal{V}}

\newcommand{\E}{\mathbb{E}}
\newcommand{\T}{\mathbb{T}}

\newcommand{\Laplace}{\Delta}

\newcommand{\Pl}{P_{\lambda}}
\newcommand{\Ql}{Q_{\lambda}}

\newcommand{\cA}{\mathcal{A}}

\newcommand{\hu}{\widehat {m}}

\newcommand{\rd}{\mathrm{d}}
\newcommand{\pd}{\partial}

\newcommand{\norm}[1]{\| #1 \|}

\usepackage{xcolor}

\usepackage[colorlinks,pdfdisplaydoctitle]{hyperref}
\newcommand{\moda}[1]{\textcolor{magenta}{#1}\index {#1}} %ANDREW
\definecolor{darkgreen}{rgb}{0,.75,0}  
\title{Accuracy and Stability of The Continuous-Time 3DVAR
Filter for The Navier-Stokes Equation}
\author{
D. Bl\"omker, 
K.J.H. Law, 
A.M. Stuart,
K. C. Zygalakis}

\begin{document}
\maketitle
\begin{abstract}
The 3DVAR filter is prototypical of
methods used to combine observed data with 
a dynamical system, online, in order to improve estimation of
the state of the system. 
Such methods are used for high dimensional data
assimilation  problems, such as those arising
in weather forecasting. To gain understanding
of filters in applications such as these, it is hence 
of interest to study their behaviour when applied to
infinite dimensional dynamical systems.
This motivates study of the problem of accuracy and 
stability of 3DVAR filters for the Navier-Stokes equation.

We work in the limit of high frequency observations
and derive continuous time filters. 
This leads to a stochastic partial differential equation (SPDE) 
for state estimation, in the form of a damped-driven
Navier-Stokes equation, with mean-reversion to the signal, and
spatially-correlated time-white noise. 
Both forward and pullback accuracy and stability results
are proved for this SPDE, showing in particular that
when enough low Fourier modes are observed,
and when the model uncertainty is larger than the data 
uncertainty in these modes (variance inflation), then 
the filter can lock
on to a small neighbourhood of the true signal, recovering from
order one initial error, if the error in the observations
modes is small. Numerical examples are given to
illustrate the theory. 
\end{abstract}

%%%%%%%%%%%%%%%%%%%%%%%%%%%%%%%%%%%%%%%% 
\section{Introduction}

%%%%%%%%%%%%%%%%%%%%%%%%%%%%%%%%%%%%
Data assimilation is
the problem of estimating the state variables of a dynamical system, given observations of the output variables. It is a challenging and fundamental problem area, of importance in a wide range of applications. A natural framework for approaching such problems is that of Bayesian statistics, since it is often the case that
the underlying model and/or the data are uncertain. However, in many real world applications, the dimensionality of the underlying model and the vast amount of available data makes the investigation of the Bayesian posterior distribution of the model state given data computationally infeasible in on-line situations. An example of such an application is the global weather prediction: the computational models used currently involve on the order of $\mathcal{O}(10^8)$ unknowns, while a large amount of partial observations of the atmosphere, currently on the order of 
$\mathcal{O}(10^6)$ per day,
are used to compensate both the uncertainty in the model and in the initial conditions.   

In situations like this practitioners typically employ some form of approximation based on both physical insight and computational expediency. There are two competing methodologies for data assimilation which are widely implemented in practice, the first being \emph{filters} \cite{kalnay2003atmospheric} and the second being \emph{variational methods} \cite{ben02}. 
In this paper we focus on the filtering approach.
Many of the filtering algorithms
implemented in practice are \emph{ad hoc} and, besides some very special cases, the theoretical understanding of their ability
to accurately and reliably estimate the state variables  is 
under-developed. Our goal here is to contribute towards
such theoretical understanding. 
We concentrate on the 3DVAR filter
which has its origin in weather forecasting \cite{lorenc1986analysis} and is prototypical of more sophisticated filters used 
today. 

The idea behind filtering is to update the posterior 
distribution of the system state sequentially at each 
observation time. This may be performed exactly for linear systems
subject to Gaussian noise: the Kalman filter \cite{harvey1991forecasting}. For the case of non-linear and non-Gaussian scenarios the particle filter \cite{doucet2001sequential} can be used and provably approximates the desired probability distribution as the number of particles increases \cite{bain2008fundamentals}. Nevertheless, standard implementations of this method perform poorly in high dimensional systems \cite{SBBA08}. Thus the development of
practical filtering algorithms for high 
dimensional dynamical systems is an active research area 
and for further insight into this subject the reader may consult 
\cite{toth1997ensemble,evensen2009data,VL09,
harlim2008filtering,majda2010mathematical,chorin2010implicit,van2010nonlinear,beskos2011stability} and references within.
Many of the methods used invoke some form of \emph{ad hoc} Gaussian approximation and the 3DVAR method which we analyze here is perhaps the simplest example of this idea. These \emph{ad hoc} filters, 3DVAR included, may also be viewed within the framework of nonlinear control theory
and thereby derived directly, without reference to the
Bayesian probabilistic interpretation; indeed this is primarily
how the algorithms were conceived. 

In this paper we will study accuracy
and stability for the 3DVAR filter.
The term {\em accuracy} refers to establishing closeness of the
filter to the true signal underlying the data, and
{\em stability} is concerned with studying
the distance between two filters,
initialized differently, but driven by the same noisy data.
Proving filter accuracy and stability results
for control systems has a long history and the paper \cite{Tarn-Rasis} is a fundamental contribution to the subject 
with results closely related to those developed here. However,
as indicated above, the high dimensionality of the problems
arising in data assimilation is a significant challenge in
the area. In order to confront this challenge we work in
an infinite dimensional setting, therby ensuring that our
results are not sensitive to dimensionality. We focus
on dissipative dynamical systems, and take the two
dimensonal Navier-Stokes equation as a prototype model
in this area. Furthermore, we study a data assimilation 
setting in which data arrives continuously in time which
is a natural setting in which to study high time-frequency
data subject to significant uncertainty. 
The study of accuracy and stability of filters for data assimilation has been a developing area over the last few years and the paper \cite{carrassi2008data} contains finite dimensional theory and numerical experiments in a variety of finite and discretized infinite dimensional systems extend the conclusions of
the theory. The paper \cite{trevisan2011chaos}
highlights the principle that, roughly speaking, 
unstable directions must be observed and assimilated into the
estimate and, more subtly, that accuracy can be improved by avoiding 
assimilation of stable directions. In particular
the papers \cite{carrassi2008data,trevisan2011chaos}
both explicitly identify the importance of observing
the unstable components of the dynamics, leading to
the notion of AUS: {\em assimilation in the unstable subspace}.
The paper \cite{lsetal} describes a theoretical analysis of 3DVAR
applied to the Navier-Stokes equation, when the data arrives
in discrete time, and
in this paper we address similar questions in the
continuous time setting; both papers include the possibility
of only partial observations in Fourier space. Taken together, 
the current paper and
\cite{lsetal} provide a significant generalization of the theory
in \cite{Tarn-Rasis} to dissipative infinite dimensional
dynamical systems prototypical of the high dimensional
problems to which filters are applied in practice; furthermore,
through studying partial observations,
they  give theoretical insight into the
idea of AUS as developed 
in \cite{carrassi2008data,trevisan2011chaos}.
The infinite dimensional
nature of the problem brings fundamental mathematical
issues into the problem, not addressed in previous finite
dimensional work. We make use of the
{\em squeezing property} 
of many dissipative dynamical systems
\cite{constantin1988navier,book:Temam1997},
including the Navier-Stokes equation, which drives
many theoretical results in this area, such
as the ergodicity studies pioneered by Mattingly
\cite{mattingly2002exponential,hairer2006ergodicity}. 
In particular our infinite dimensional analysis is motivated
by the theory developed in \cite{olson2003determining} and \cite{hayden2011discrete}, which are the first 
papers to study data assimilation directly through PDE analysis, using ideas from the theory of  determining modes in infinite dimensional dynamical systems.  However, in contrast to those papers, here we allow for noisy observations, and provide a methodology
that opens up the possibility of studying more general Gaussian 
approximate filters such as the Ensemble and the Extended Kalman filter (EnKF and ExKF).

Our point of departure for analysis is an ordinary differential equation (ODE) in a Banach space. 
Working in the limit of high frequency observations  we formally derive continuous time filters. 
This leads to a stochastic differential equation for 
state estimation, combining the original dynamics with
extra terms indcuing mean reversion to the noisily observed
signal. In the particular case of the Navier-Stokes equation
we get a stochastic PDE (SPDE) with 
additional mean-reversion term,
driven by spatially-correlated time-white noise.
This SPDE is central to our analysis as it is used to prove 
accuracy and stability results for the 3DVAR filter. In particular, 
 in the case when enough of the low modes of the Navier-Stokes equation 
are observed and the model has larger uncertainty than the data in these low modes, 
a situation known to practitioners as variance inflation, then the filter can lock on to a small neighbourhood of the true signal,
recovering from the initial error, if the error in the
observed modes is small.
The results are formulated in terms of the theory of random and
stochastic dynamical systems \cite{Arn98}, 
and both forward and pullback type results are proved, leading
to a variety of probabilistic accuracy and stability results,
in the mean square, probability and almost sure senses.

The paper is organised as follows. In Section 2 we derive the continuous-time limit of the 3DVAR filter applied to a general ODE in a Banach space,  by considering the limit of high frequency observations. In Section 3, we focus on the 2D Navier-Stokes equations and present the continuous time 3DVAR filter
within this setting. Sections 4 
and 5 are devoted, respectively, to results 
concerning forward accuracy and stability as well as pullback 
accuracy and stability, for the filter when applied to the
Navier-Stokes equation. 
In Section 6, we present various numerical investigations 
that corroborate our theoretical results. Finally 
in Section 7 we present conclusions.

%%%%%%%%%%%%%%%%%%%%%%%%%%%%%%%%%%%%%

\section{Continuous-Time Limit of 3DVAR}

%%%%%%%%%%%%%%%%%%%%%%%%%%%%%%%%%%%%%%%%

Consider $u$ satisfying the following ODE in a Banach space $X:$
\begin{equation} \label{e:model}
 \frac{du}{dt}=\cf(u), \quad u(0)=u_{0}\;.
\end{equation}
Our aim is to study online filters which combine
knowledge of this dynamical system with noisy observations
of $u_n=u(nh)$ to estimate the state of the system.
This is particularly important in applications where
$u_0$ is not known exactly, and the noisy data can
be used to compensate for this lack of initial
knowledge of the system state.

In this section we study approximate Gaussian filters
in the high frequency limit, leading to stochastic
differential equations which combine the dynamical
system with data to estimate the state.
As the formal derivation of continuous time filters in this
section is independent of the precise model under consideration, 
we employ the general framework of \eqref{e:model}.
We make some general observations, relating to a broad family
of approximate Gaussian filters, but focus mainly on 3DVAR. 
In subsequent sections, where we study stability and accuracy 
of the filter, we focus exclusively on 3DVAR,
and work in the context of the $2$D 
incompressible Navier-Stokes equation, as this is prototypical 
of dissipative semilinear partial differential equations.

%However, we will focus
%mainly on 3DVAR. Indeed in subsequent sections we will
%focus exclusively on 3DVAR, using the Navier-Stokes equation
%as a running example throughout.

%%%%%%%%%%%%%%%%%%%%%%%%%%%%%%%%%%%%%%%%

\subsection{Set Up - The Filtering Problem}
\label{ssec:setup}

%%%%%%%%%%%%%%%%%%%%%%%%%%%%%%%%%%%%%%%%

We assume that $u_0 \sim N(\hat{m}_{0},\hat{C}_{0})$
so that the initial data is only known statistically. 
The objective is to update the estimate of the state 
of the system sequentially in time, based
on data received sequentially in time.
We define the flow-map $\Psi:X \times \bbR^+ \to X$ so that
the solution to \eqref{e:model} is 
$u(t)=\Psi(u_{0};t)$. Let $H$ denote a linear operator
from $X$ into another Banach space $Y$, and assume that we
observe $Hu$ at equally spaced time intervals: 
\begin{equation} \label{e:data} 
y_{n}=H\Psi(u_{0};nh)+\eta_{n}. 
\end{equation}
Here $\{\eta_{n}\}_{n \in \bbN}$ is an i.i.d sequence,
independent of $u_0$, with $\eta_{1} \sim N(0,\Gamma).$   If we write $u_n=\Psi(u_0;nh)$, then
\begin{equation} \label{e:map}
u_{n+1}=\Psi(u_n;h),
\end{equation}
and
\begin{equation} \label{e:like}
y_n|u_n \sim N\bigl(Hu_n,\Gamma).
\end{equation}
We denote the accumulated data up to the time $n$ by 
\[
Y_{n}=\{y_{i}\}^{n}_{i=1}.
\]
Our aim is to find $\IP(u_n|Y_n).$ 

We will make the Gaussian ansatz that 
\begin{equation} \label{e:assum1}
\IP(u_{n}|Y_{n}) \simeq N(\hat{m}_{n},\hat{C}_{n}). 
\end{equation}
The key question in designing an approximate Gaussian filter, then, is to find an update rule of the form
\begin{equation} \label{e:update}
(\hat{m}_{n},\hat{C}_{n}) \mapsto  (\hat{m}_{n+1},\hat{C}_{n+1})
\end{equation}
Because of the linear form of the observations in \eqref{e:data}, together with the fact that the noise is mean zero-Gaussian,
 this update rule is determined directly if we impose 
a further Gaussian ansatz, now on the distribution 
of $u_{n+1}$ given $Y_{n}:$ 
\begin{equation} \label{e:assum2}
u_{n+1}|Y_{n} \sim N(m_{n+1},C_{n+1}) 
\end{equation}

With this in mind, the update \eqref{e:update}
is usually split into two parts. The first,
{\em prediction} (or forecast), step is the map
\begin{equation} \label{e:predict}
(\hat{m}_{n},\hat{C}_{n}) \mapsto  (m_{n+1},C_{n+1})
\end{equation}
The second, {\em analysis}, step is 
\begin{equation} \label{e:analysis}
(m_{n+1},C_{n+1}) \mapsto  (\hat{m}_{n+1},\hat{C}_{n+1}).
\end{equation}

For the prediction step
we will simply \emph{impose} the approximation 
\eqref{e:assum2} with 
\begin{equation} \label{e:mean_update}
m_{n+1}=\Psi(\hat{m}_{n};h), 
\end{equation}
while the choice of $C_{n+1}$ will depend on the choice of the specific filter. For the analysis step, 
assumptions \eqref{e:like}, \eqref{e:assum2} imply that 
\begin{equation}
 u_{n+1}|Y_{n+1} \sim N(\hat{m}_{n+1},\hat{C}_{n+1}) 
\end{equation}
and an application of Bayes rule, as applied in the
standard Kalman filter update \cite{harvey1991forecasting}, 
and using \eqref{e:mean_update}, gives us the nonlinear 
map \eqref{e:update} in the form
\begin{subequations} \label{e:normal_update}
\begin{eqnarray} 
\hat{C}_{n+1} &=& C_{n+1}-C_{n+1}H^*(\Gamma+HC_{n+1}H^*)^{-1}HC_{n+1} \\
\hat{m}_{n+1} &=& \Psi(\hat{m}_{n};h)+C_{n+1}H^*(\Gamma+HC_{n+1}H^*)^{-1}(y_{n+1}-Hm_{n+1})
\end{eqnarray}
\end{subequations}
The mean $\hat{m}_{n+1}$ is an element of the Banach space $X$,
and $\hat{C}_{n+1}$ is a linear 
symmetric and non-negative operator from $X$ into itself.

\begin{comment}
Alternatively, if $\Gamma$ and $C_{n+1}$ are positive-definite,
then this may be written as 
\begin{subequations} \label{e:inverse_update}
\begin{eqnarray} 
\hat{C}^{-1}_{n+1} &=& C^{-1}_{n+1}+\Gamma^{-1}, \\
\hat{C}^{-1}_{n+1}\hat{m}_{n+1} &=& C^{-1}_{n+1}m_{n+1}+\Gamma^{-1}y_{n+1}.  
\end{eqnarray}
\end{subequations}
\end{comment}

%$%%%%%%%%%%%%%%%%%%%%%%%%%%%%%%%%%%%%%%%%%%

\subsection{Derivation of The Continuous-Time Limit}
\label{sec:deriv}

%$%%%%%%%%%%%%%%%%%%%%%%%%%%%%%%%%%%%%%%%%%%

Together equations \eqref{e:mean_update} and
\eqref{e:normal_update}, which are generic 
for {\em any} approximate Gaussian filter, specify
the update for the mean once the equation determining 
$C_{n+1}$ is defined. We proceed to 
derive a continuous-time limit for the mean,
in this general setting, assuming 
that $C_n$ arises as an approximation of a continuous
process $C(t)$ evaluated at $t=nh$, so that $C_n \approx C(nh)$, 
and that $h \ll 1$.
Throughout we will assume that 
$\Gamma=h^{-1}\Gamma_{0}$.
This scaling implies that the noise variance is inversely
proportional to the time between observations and is
the relationship which gives a nontrivial stochastic
limit as $h \to 0.$

With these scaling assumptions 
equation (\ref{e:normal_update}b) becomes
\begin{equation*}
\hat{m}_{n+1}=\Psi(\hat{m}_{n};h)+hC_{n+1}H^*(\Gamma_{0}+
hHC_{n+1}H^*)^{-1}\bigl(y_{n+1}-H\Psi(\hat{m}_{n};h)\bigr).
\end{equation*} 
Thus
\begin{equation*}
\frac{\hat{m}_{n+1}-\hat{m}_{n}}{h} 
= \frac{\Psi(\hat{m}_{n};h)-\hat{m}_{n}}{h}+C_{n+1}H^*(\Gamma_{0}+hHC_{n+1}H^*)^{-1}\bigl(y_{n+1}-H\Psi(\hat{m}_{n};h)\bigr).
\end{equation*}
If we define the sequence $\{z_{n}\}_{n \in \mathbb{Z}^{+}}$ by 
\[
z_{n+1}=z_{n}+hy_{n+1}, \quad z_0=0\;,
\]
then we can rewrite the previous equation as 
\begin{align} \label{e:prelim}
\frac{\hat{m}_{n+1}-\hat{m}_{n}}{h} =& \frac{\Psi(\hat{m}_{n};h)-\hat{m}_{n}}{h}\\
&\quad\quad\quad
+C_{n+1}H^*(\Gamma_{0}+hHC_{n+1}H^*)^{-1}\left(\frac{z_{n+1}-z_{n}}{h}-H\Psi(\hat{m}_{n};h)\right). \notag 
\end{align}
Note that
$$
\Psi(\hat{m}_{n};h)=\hat{m}_{n}+h\cf(\hat{m}_{n})+\mathcal{O}(h^{2}).
$$ 

This is an Euler-Maruyama-like discretization of
a stochastic differential equation which,
if we pass to the limit of $h \rightarrow 0$ in
\eqref{e:prelim}, 
noting that we have assumed that $C_n \approx C(nh)$ 
for some continuous covariance process, is seen to be 

\begin{equation} \label{e:basic}
\frac{d\hat{m}}{dt}=\cf(\hat{m})+CH^*\Gamma^{-1}_{0}\left(\frac{dz}{dt}-H\hat{m}\right), \quad \hat{m}(0)=\hat{m}_{0}. 
\end{equation}

Equation \eqref{e:basic} is similar to the observer equation 
in the nonlinear  control literature \cite{Tarn-Rasis}.
Our objective in this paper is to study the stability and
accuracy properties of this stochastic model. Here stability
refers to the contraction of two different trajectories of
the filter \eqref{e:basic}, started at two different points, but
driven by the same observed data; and accuracy refers to estimating 
the difference between the true trajectory of \eqref{e:model}
which underlies the data, and the output of the filter
\eqref{e:basic}.
Similar questions are studied in finite dimensions in
\cite{Tarn-Rasis}. However, the infinite dimensional nature
of our problem, coupled with the fact that we study situations
where the state is only partially observed ($H$ is not invertible
on $X$) mean that new techniques of analysis are required,
building on the theory of semilinear dissipative PDEs
and infinite dimensional dynamical systems.

We now express the observation signal $z$ in terms of the 
truth $u$ in order to facilitate study of filter 
stability and accuracy. In particular, we have that 
\[
\left( \frac{z_{n+1}-z_{n}}{h} \right)=y_{n+1}=Hu_{n+1}+\frac{\sqrt{\Gamma_{0}}}{\sqrt{h}}\Delta w_{n+1}, 
\]
where $\{ \Delta w_{n}\}_{n \in \bbN}$ is an i.i.d sequence and $\Delta w_{1} \sim N(0,I)$ in $Y$. This corresponds to the 
Euler-Maruyama discretization of the SDE
\begin{equation}
\label{e:obs}
\frac{dz}{dt}=Hu+\sqrt{\Gamma_{0}}\frac{dW}{dt}, \quad z(0)=0. 
\end{equation}
Expressed in terms of the true signal $u$,
equation \eqref{e:basic} becomes
\begin{equation} \label{e:3dvar}
\frac{d\hat{m}}{dt}= \cf(\hat{m})+CH^*\Gamma^{-1}_{0}H\left(u-\hat{m}\right)+ \moda{C}H^*\Gamma^{-1/2}_{0}\frac{dW}{dt}.
\end{equation}

We complete the study of the continuous limit
with the specific example of 3DVAR.
This is the simplest filter of all in which the
prediction step is found by simply setting
$C_{n+1}=\hat{C}$ for some {\em fixed} 
covariance operator $\hat{C}$, independent of $n$.
Then equation \eqref{e:normal_update} shows that
${\hat C}_{n+1}=\hat{C}+{\cal O}(h)$ and
we deduce that the limiting covariance is simply constant:
$C(t)=\hat{C}(t)=\hat{C}$ for all $t\ge 0$. 
The present work will focus on this case and hence study
\eqref{e:3dvar} in the case where $C=\hat{C}$, 
a constant in time.

%%%%%%%%%%%%%%%%%%%%%%%%%%%%%%%%%%%%%%%%%%%%

\section{Continuous-Time 3DVAR for Navier-Stokes}
\label{sec:nse}

%%%%%%%%%%%%%%%%%%%%%%%%%%%%%%%%%%%%%%%%%%%%

In this section we describe application of the 3DVAR
algorithm to the two dimensional Navier-Stokes equation.
This will form the focus of the remainder of the paper.
In subsection \ref{ssec:nsef} we describe the 
forward model itself, namely we specify equation 
\eqref{e:model}, and then in subsection \ref{ssec}
we describe how data is incorporated into the model,
and specify equation \eqref{e:3dvar}, and the choices
of the (constant in time)
operators $C=\hat{C}$ and $\Gamma_0$ which appear
in it.

\subsection{Forward Model}
\label{ssec:nsef}

Let $\T^{2}$ denote the two-dimensional
torus of side $L:$ $[0,L) \times [0,L)$
with periodic boundary conditions. We
consider the equations
%\begin{subequations}
\begin{equation*}
\begin{array}{ccc}
%\begin{align}
\pd_{t}u(x, t) - \nu \Laplace u(x, t)
 + u(x, t) \cdot \nabla u(x, t) + \nabla p(x, t) 
&=& f(x)  
\\
\nabla \cdot u(x, t) &=& 0 
\\
u(x, 0) &=& u_{0}(x) 
\end{array}
%\label{eq:NSE}
\end{equation*}
for all $x \in \T^{2}$ and $t\in(0, \infty)$. 
%\end{subequations}
Here $u \colon \T^{2} \times (0, \infty) \to \R^{2}$ is a time-dependent vector field representing the velocity, $p \colon \T^{2} \times (0,\infty) \to \R$ is a time-dependent scalar field representing the pressure and $f \colon \T^{2} \to \R^{2}$ 
is a vector field representing the forcing which we take as 
time-independent for simplicity. The parameter $\nu$ 
represents the viscosity. We assume throughout that $u_0$
and $f$ have average zero over $\T^2$; it then follows
that $u(\cdot,t)$ has average zero over $\T^2$ for all
$t>0$. 

Define
%\[
$${\mathsf T}:= \left\{ {\rm {trigonometric\,polynomials\,}} 
u:\T^2 \to {\mathbb R}^2\,\Bigl|\, \nabla \cdot u = 0, \,\int_{\T^{2}} u(x) \, \rd x = 0 \right\}
$$
%\] 
and $\cH$ as the closure of ${\mathsf T}$ with respect to the
norm in $(L^{2}(\T^{2}))^{2} = L^{2}(\T^{2},\R^2)$.

We let $P:(L^{2}(\T^{2}))^{2} 
\to \cH$ denote the Leray-Helmholtz orthogonal projector. 
Given $k = (k_{1}, k_{2})^{\mathrm{T}}$, define $k^{\perp} := (k_{2}, -k_{1})^{\mathrm{T}}$. Then an orthonormal basis for
$\cH$ is given by $\psi_{k} \colon \R^{2} \to \R^{2}$, where 
\begin{equation}
\label{eq:fb}
\psi_{k} (x) := \frac{k^{\perp}}{|k|} \exp\Bigl(\frac{2 \pi i k \cdot x}{L}\Bigr)
\end{equation}
for $k \in \Z^{2} \setminus \{0\}$. 
Thus for $u \in \cH$ we may write
$$
u(x,t) = \sum_{k \in \Z^{2} \setminus \{0\}} u_{k}(t) \psi_{k}(x)
$$
where, since $u$ is a real-valued function, we have the 
reality constraint $u_{-k} = - \overline{u_{k}}.$
We define the projection operators $\Pl: \cH \to \cH$ 
and $\Ql:\cH \to \cH$ for 
$\lambda \in\mathbb{N}\cup\{\infty\}$ 
by
$$\Pl u(x,t)  = \sum_{|2\pi k|^2 <\lambda L^2} u_{k}(t) \psi_{k}(x),
\quad \Ql=I-\Pl.$$
Below we will choose the observation operator $H$ to be $\Pl.$

We define $A = -\frac{L^2}{4\pi^2} P \Laplace$, 
the Stokes operator, and, for every $s \in \R$, 
define the Hilbert spaces $\cH^s$ to be the domain of $A^{s/2}.$ 
We note that $A$ is diagonalized in $\cH$ in the basis comprised
of the $\{\psi_k\}_{k \in {\Z}^2\backslash\{0\}}$ 
and that, with the normalization employed
here, the smallest eigenvalue of $A$ is $\lambda_1=1.$
We use the norm 
$\|\cdot\|_{s}^2:=\langle \cdot, A^s \cdot \rangle$, 
the abbreviated notation $\norm{u}$ for the norm on $\cV:=\cH^1$,
and $|\cdot|$ for the norm on $\cH:=\cH^0$. 

Applying the projection $P$ to the Navier-Stokes
equation we may write it as an ODE in $\cH$:
\begin{equation}
\frac{\rd u}{\rd t} + \delta Au + \cB(u, u) = f, \quad u(0)=u_0.
\label{eq:nse}
\end{equation}
Here $\delta=4\pi^2\nu/L^2$ 
and the term $\cB(u,v)$ 
is the {\it symmetric} 
bilinear form defined by 
$$\cB(u,v) = \frac12 P(u \cdot \nabla v)+ \frac12 P(v \cdot \nabla u)$$
for all $u,v\in\cV$.
Finally, with abuse of notation, $f$ is the original forcing, 
projected into $\cH$. Equation \eqref{eq:nse} is in the form 
of equation \eqref{e:model} with
\begin{equation}
\label{eq:F}
\cf(u)=-\delta Au-\cB(u,u)+f.
\end{equation}

See \cite{constantin1988navier} 
for details of this formulation of the Navier-Stokes
equation as an ODE in $\cH$.
The following proposition
is a classical result which implies the existence
of a dissipative semigroup for the ODE \eqref{eq:nse}. 
See Theorems 9.5 and 12.5 in \cite{book:Robinson2001}
for a concise overview and \cite{temam1995navier,book:Temam1997}
for further details.

\begin{proposition} 
\label{prop:1}
Assume that $u_0 \in \cH^1$ and $f \in \cH$.
Then \eqref{eq:nse} has a unique strong solution on $t
\in [0,T]$ for any $T>0:$
$$u \in L^{\infty}\bigl((0,T);\cH^1\bigr)\cap L^{2}\bigl((0,T);\cH^2\bigr),\quad \frac{du}{dt} \in L^{2}\bigl((0,T);\cH\bigr).$$
Furthermore the equation has a global attractor $\cA$
and there is $R\in(0,\infty)$ such that, 
if $u_0 \in \cA$, then the solution from this initial
condition exists for all $t \in \bbR$ and $\sup_{t \in \bbR}\|u(t)\|^2 = R.$
\end{proposition}

We let $\{\Psi(\cdot,t): \cH^1 \to \cH^1\}_{t \ge 0}$ 
denote the semigroup of 
solution operators for the equation \eqref{eq:nse} through $t$
time units. We note that by
working with weak solutions,
$\Psi(\cdot,t)$ can be extended to act on larger spaces $\cH^s$,
with $s \in [0,1)$, under the same assumption on $f$;
see Theorem 9.4 in \cite{book:Robinson2001}.  

%%%%%%%%%%%%%%%%%%%%%%%%%%%%%%%%%%%%%%%%

\subsection{3DVAR}
\label{ssec}

%%%%%%%%%%%%%%%%%%%%%%%%%%%%%%%%%%%%%%%%%%%%

We apply the analysis of the previous section to write
down the continuous time 3DVAR filter,
namely \eqref{e:3dvar} with $C(t)={\hat C}$ constant in time,
for the Navier-Stokes equation. We take $X=\cH$ and
throughout we assume that the data is found by observing
$\Pl u$ at discrete times, so that 
$H^*=H=\Pl$ and $Y=\Pl \cH$.
We assume that $A$, $\Gamma_0$ and $\hat{C}$
commute and, for simplicity of presentation, suppose that
\begin{equation}
\label{e:assumCG}
\hat{C} = \omega \sigma_0^2 A^{-2\zeta},
\quad \Gamma_0= \sigma_0^2 A^{-2\beta}\Pl.
\end{equation}
We set $\alpha=\zeta-\beta$.
These assumptions correspond to those made in \cite{lsetal}
where discrete time filters are studied.
Note that $A^{s/2}$ is defined on $\cH^s$; it
is also defined on $Y$ for every $s \in \bbR$,
provided that $\lambda$ is finite.

From equations \eqref{e:3dvar}, using \eqref{eq:F} and
the choices for $\hat{C}$ and $\Gamma_0$ we obtain
\begin{equation}
\frac{\rd \hu}{\rd t} 
+ \delta A\hu 
+ \cB(\hu, \hu) 
+\omega A^{-2\alpha}\Pl(\hu-u)
= f
+\omega \sigma_0A^{-2\alpha-\beta}\Pl\frac{\rd W}{\rd t},
\quad \hu(0)=\hu_0
\label{eq:nse2}
\end{equation}
where $W$ is a cylindrical Brownian motion in $Y$.
In the following we consider the cases of finite $\lambda$,
where the data is in a finite dimensional subspace of $\cH$,
and infinite $\lambda$, where $\Pl=I$ and the whole solution
is observed.

\begin{lem}
 \label{lem:SC}
For $\lambda=\infty$
assume that $4\alpha+2\beta>1.$ 
Then the stochastic convolution 
\[
W_A(t) = \int_0^t e^{\delta(t-s)A} A^{-2\alpha-\beta} \Pl dW(s)
\]
has a continuous version in $C^0([0,T], \cV)$ 
with all moments $\mathbb{E} \sup_{[0,T]} \|W_A\|^p$ finite
for all $T>0$ and $p>1$. 
\end{lem}

\begin{proof}
If $\lambda<\infty$ then the covariance of the
driving noise is automatically trace-class
as it is finite dimensional;
since $4\alpha+2\beta>1$ it follows that 
the covariance of the driving noise is also
trace-class when $\lambda=\infty.$
The desired result follows from Theorem 5.16
in \cite{Dap-Z}.  For the moments see \cite[(5.23)]{Dap-Z}.
\end{proof}

It is only in the case of full observations 
(i.e., $\lambda=\infty$) that we need the additional 
regularity condition $4\alpha+2\beta>1$. This
may be rewritten as $\zeta>\frac14+\frac12\beta$ and relates the
rate of decay, in Fourier space, of the model variance 
to the observational variance. Although a key driver
for our accuracy and stability results (see Remark \ref{rem:eps}
below) will be variance inflation, meaning that the observational
variance is smaller than the model variance in the low Fourier modes,
this condition on $\zeta$ allows regimes in which, for high
Fourier modes, the situation is reversed.

\begin{pro}
 \label{prop:ex-S2DNS}
Assume that $u_0 \in \cA$
and let $u$ be the corresponding solution of (\ref{eq:nse}) 
on  the global attractor $\cA$. 
For $\lambda=\infty$
suppose $4\alpha+2\beta>1$ and $\alpha>-\frac12.$
Then for any initial condition $\hat{m}(0) \in \cH$ 
there is a stochastic process $\hat{m}$
which is the unique strong solution
of (\ref{eq:nse2}) in the spaces 
\[
\hat{m} \in 
L^2\bigl((0,T),\cV\bigr) \cap  C^0\bigl([0,T],\cH\bigr)
%{L^\infty\bigl((0,T),\cH\bigr) \cap L^2\bigl((0,T),\cV\bigr)}
%\quad\text{and}\quad 
%\hat{m} \in C^0\bigl([0,T],\cH\bigr) 
\] 
for all $T>0$.
Moreover,
\[
 \EX\|\hat{m}\|^2_{L^\infty\bigl((0,T),\cH\bigr)} +  
\EX\|\hat{m}\|^2_{L^2\bigl((0,T),\cV\bigr)} <\infty. 
\]
To be more precise
\[
\hat{m}\in L^2\Bigl(\Omega, C^0_\loc\bigl([0,\infty),\cH \bigr)\Bigr) \cap L^2\Bigl(\Omega,L^2_\loc\bigl([0,\infty),\cV\bigr)\Bigr)\;.
\]
\end{pro}

\begin{proof}
The proof of this theorem is well known
without the function $u$
and the additional linear term.
See for example Theorem 15.3.1 in the book \cite{dap2} 
using fixed point arguments 
based on the mild solution. Another reference is 
\cite[Theorem 3.1]{Fl:94} based on spectral Galerkin methods.
See also \cite{FlMa:95} or \cite{Sch:87}.
Nevertheless, our theorem is a straightforward modification of their arguments.
For simplicity of presentation we refrain from giving 
a detailed argument here.

The existence and uniqueness is established either by 
Galerkin methods or fixed-point arguments.
The continuity of solutions follows from the standard fixed-point arguments 
for the mild formulation in the space $C^0\bigl([0,T],\cH\bigr)$.

Finally, as we assume the covariance of the Wiener process 
to be trace-class, the bounds on the moments are a straightforward 
application of It\^o's formula 
(cf. \cite[Theorem 4.17]{Dap-Z})
to $|\hat{m}|^2$ and to $\|\hat{m}\|^2$, 
in order to derive standard a-priori estimates. 
This is very similar to the method of proof 
that we use to study mean square stability 
in Section \ref{sec41}.

The additional linear term $\omega A^{-2\alpha}\Pl\hu$ does not change the result in
any substantive fashion.
If $\lambda<\infty$ then the proof is essentially identical,
as the additional term is a lower order perturbation of the Stokes operator.

If $\lambda=\infty$ then minor modifications of the 
proof are necessary, but do not change the proof significantly.
This is since, for $\alpha>-\frac12$, 
the additional term 
$\omega A^{-2\alpha}\hat{m}$ 
is a compact perturbation of the 
Stokes operator. 

The additional forcing term,
depending on $u$, is always sufficiently regular 
for our argument,
as we assume $u$ to be on the attractor (see Proposition \ref{prop:1}). 
%Solutions on the attractor are even smooth,
%provided the forcing  $f$ is smooth.
\end{proof}

\begin{rem}
For $\lambda=\infty$
it is possible to extend the preceding result 
to other ranges of $\alpha$, but this will change the proof. 
Hence, for simplicity, 
for $\lambda=\infty$ we always
assume that
$\alpha>-\frac12$.

We comment later on 
the fact that the solutions to (\ref{eq:nse2}) 
generate a stochastic dynamical system. 
As we need two-sided Wiener-processes for this 
we postpone the discussion to Section \ref{sec:PB}. 
\end{rem}
%\note{KCZ:Shouldn't the reference here be \ref{eq:nse2}? KL: agreed, changed.}
%%%%%%%%%%%%%%%%%%%%%%%%%%%%%%%%%%%%%%%%%%%%

\section{Forward Accuracy and Stability} % of the Continuous Model}

%%%%%%%%%%%%%%%%%%%%%%%%%%%%%%%%%%%%%%%%%%%%

We wish to study conditions under which two filters, 
starting from different points but driven by the same 
observations, converge (stability); and conditions under
which the filter will, asymptotically, track the 
true signal (accuracy). 
Establishing such results has been
the object of study in control theory for some time,
and the paper \cite{Tarn-Rasis} contains 
foundational work in both the discrete and continuous time
settings. However the
infinite dimensional nature of the problem at hand brings
significant new challenges to the analysis. The key idea
driving the proofs is that, although the Navier-Stokes equations
themselves may admit exponentially diverging trajectories,
the observations can counteract this instability, provided
the observation space is large enough. Roughly speaking
the exponential divergence of the Navier-Stokes equations
is dominated by a finite set low Fourier modes, whilst the rest
of the space contracts. If the observations provide
information about enough of the low Fourier modes, 
then this can counteract the instability. 
This basic idea underlies the accuracy and stability
results proved in subsections \ref{sec41} and \ref{ssec:z}.

A key technical estimate in what follows is the following
(see \cite{book:Temam2001}): 
\begin{lem}
\label{lem:B1}
For the symmetric bilinear map
$$\cB(u,v) = \tfrac12 P(u \cdot \nabla v)+ \tfrac12 P(v \cdot \nabla u)
$$
there 
is constant $K'\ge1$ such that
for all $v,w \in \cV$
\begin{equation}
\langle \cB(v,v)-\cB(w,w), v-w\rangle \leq K'\|w\| \|v-w\| \cdot |v-w|,
\label{eq:Best0}
\end{equation}
\begin{equation}
\label{eq:Best1}
|\langle \cB(w,v),v \rangle| \le K' \|w\| \|v\| |v|
\quad\text{and}\quad
|\langle \cB(w,v),z \rangle| \le K' \|v\| \|w\| \|z\|\;.
\end{equation}
Furthermore, for all $v \in \cV$, 
\begin{equation}
\langle \cB(v,v), v \rangle =0
\label{eq:Best2}
\end{equation}
\end{lem}

%%%%%%%%%%%%%%%%%%%%%%%%%%%%%%%%%%%%%%%%%%%%%
%

Notice that (\ref{eq:Best0}) implies that, for $K=(K')^2/\delta$,
\begin{equation}
\langle \cB(v,v)-\cB(w,w), v-w\rangle \leq \tfrac12 K\|w\|^2|v-w|^2 +\tfrac12\delta \|v-w\|^2,
\quad \forall v,w \in \cV.
\label{eq:Best}
\end{equation}
This estimate will be used to control the possible
exponential divergence of Navier-Stokes trajectories
which needs to be compensated for by means of observations.

%%%%%%%%%%%%%%%%%%%%%%%%%%%%%%%%%%%%%%%%%%%%

\begin{proof}[Proof of Lemma \ref{lem:B1}.]
We give a brief overview of the main ideas required
to prove this well known result.
First notice that the assumption $K'\ge1$ is without loss of generality.
We need this later for simplicity of presentation. 

Then (\ref{eq:Best0}) is a direct consequence of (\ref{eq:Best1})
and (\ref{eq:Best2}), by
using the identity 
$$
\cB(v,v)-\cB(w,w) = \cB(v+w, v-w) = \cB(v-w, v-w)+ 2 \cB(w, v-w)
$$
For simplicity of presentation, we use the same constant in (\ref{eq:Best0})
and (\ref{eq:Best1}).

For Navier-Stokes it is well-known that
$\langle (w\cdot\nabla)v,v \rangle=0$, as the divergence of $w$ is $0$.
Thus there is constant $c_1$ such that
$$
2|\langle \cB(w,v),v \rangle| 
= \langle (v\cdot\nabla)w,v \rangle| 
\le c_1\|w\| |v|_{L^4}^2\;.
$$
Since, in two dimensions $H^{1/2} = D(A^{1/4})$ is embedded 
into $L^4$, there is constant $c_2$ such that
$
|\langle \cB(w,v),z \rangle| \le c_2 \|w\| \|v\|_{H^{\frac12}} 
\|z\|_{H^{\frac12}}\;.$ The first result in (\ref{eq:Best1}) 
then follows from the
interpolation inequality $\|v\|^2_{H^{\frac12}} \le c |v| \|v\|$ 
and the second from the embedding
$\|v\|_{H^{\frac12}} \le c \|v\|$.
\end{proof}

Finally, in the following a key role will be played
by the constant $\gamma$ defined as follows: 
\begin{ass}
Let $\gamma$ be
the largest positive constant such that
\begin{equation}
 \label{e:bass2a}
\frac12 \gamma|h|^2 \le \langle  \omega A^{-2\alpha}\Pl h , h \rangle  + \frac12 \delta \| h \|^2
\qquad \text{for all }h\in \cV.
\end{equation}
\end{ass}

It is clear that such a $\gamma$ always exists, and
indeed that 
$\gamma \ge \delta$, as one has $\langle A^{-2\alpha}\Pl h , h \rangle\ge 0 $.  We will study how
$\gamma$ depends on $\lambda$ and $\omega$ in subsequent
discussions where we show that, by choosing $\lambda$
and $\omega$ large enough, $\gamma$ can be made
arbitrarily large.

%%%%%%%%%%%%%%%%%%%%%%%%%%%%%%%%%%%%%%%

\subsection{Forward Mean Square Accuracy}
\label{sec41} 

%%%%%%%%%%%%%%%%%%%%%%%%%%%%%%%%%%%%%%%%%

%%%%%%%%%%%%%%%%%%%%%%%%%%%%%%%%%%%%%%%%%%%%%%%%%%%%%%%%

\begin{theo}[\bf{}Accuracy]
\label{thm:mainacc}
Let $\hat{m}$ solve \eqref{eq:nse2}, and let 
$u$ solve \eqref{eq:nse} with initial condition
on the global attractor $\cA$.
For $\lambda=\infty$
assume $4\alpha+2\beta>1$ and $\alpha>-\frac12$. 
Suppose that $\gamma$, the largest positive number
such that \eqref{e:bass2a} holds, satisfies 
\[\gamma=KR+\gamma_0 \qquad\text{for some } \gamma_0>0, 
\]
where $K$ is the constant appearing in \eqref{eq:Best}
and $R$, recall, is defined by $R=\sup_{t \in \bbR} \|u(t)\|^2$.
Then 
\[
\EX|\hat{m}(t)-u(t)|^2 
\leq \mathrm{e}^{-\gamma_0 t}|\hat{m}(0)-u(0)|^2+\omega^2\sigma_0^2\int_0^t
\mathrm{e}^{-\gamma_0(t-s)} 
 {\trace}_{\cH}\bigl(A^{-4\alpha-2\beta}\Pl\bigr) ds.
\]
As a consequence 
\[
\limsup_{t\to\infty} \EX|\hat{m}(t)-u(t)|^2 \leq 
\tfrac1{\gamma_0}\omega^2\sigma_0^2 {\trace}_{\cH}\bigl(A^{-4\alpha-2\beta}\Pl\bigr).
\]
\end{theo}

%%%%%%%%%%%%%%%%%%%%%%%%%%%%%%%%%%%%%%%%%%%%%%%%%%%%%%%

\begin{proof}
Define the error $e=\hat{m}-u$
and subtract equation \eqref{eq:nse} from \eqref{eq:nse2}
to obtain 
\begin{equation*}  
de +\delta Ae= \Bigl(\cB(u,u)-\cB(\hat{m},\hat{m})-\omega
A^{-2\alpha}\Pl e \Bigr)dt+\omega \sigma_0A^{-2\alpha-\beta}\Pl dW\;.
\end{equation*}
Using the It\^o formula from Theorem 4.17
of \cite{Dap-Z}, together with \eqref{eq:Best}, yields
\begin{align*}
\tfrac12 d |e |^2 \leq&
\Bigl(-\tfrac12 \delta \| e \|^2 + \tfrac12 K \| u(t) \|^2  | e |^2 
-\langle \omega A^{-2\alpha}\Pl e , e \rangle \Bigr)dt\\ 
&\qquad \qquad
+ \langle e ,\omega \sigma_0 A^{-2\alpha-\beta} \Pl dW\rangle
+ \tfrac12 \omega^2\sigma_0^2 {\trace}_{\cH}\bigl(A^{-4\alpha-2\beta}\Pl
\bigr) dt
 \;.
\end{align*}
Here we have used the fact that 
the projection $\Pl$ and $A$ commute.
% the trace of $A^{s}$ in $\cH$
%dominates the trace in $Y$.
Applying \eqref{e:bass2a} and taking expectations we
obtain
\[
\frac{d}{dt} \EX|e(t)|^2
\leq
-\bigl(\gamma-K\|u(t)\|^2\bigr)\cdot\EX| e(t) |^2 
+ \omega^2\sigma_0^2 {\trace}_{\cH}\bigl(A^{-4\alpha-2\beta}\Pl \bigr) 
 \;,
\]
But $\sup_{t \ge 0} \|u(t)\|^2 =R<\infty$ by Proposition \ref{prop:1}
and hence, by assumption on $\gamma$,
\[
\frac{d}{dt} \EX|e(t)|^2
\leq
-\gamma_0\cdot \EX| e(t) |^2 
+ \omega^2\sigma_0^2 {\trace}_{\cH}\bigl(A^{-4\alpha-2\beta}\Pl
\bigr).
 \;,
\]
The result follows from a Gronwall argument. 
\end{proof}

\begin{rem}
\label{rem:eps}
We now briefly discuss the choice of parameters
to ensure satisfaction of the conditions of Theorem
\ref{thm:mainacc}, and its implications.
To this end, notice that $K$
and $R$ are independent of the parameters of the
filter, being determined entirely by the
Navier-Stokes equation \eqref{eq:nse} itself.
To apply the theorem we need to ensure that
$\gamma$ defined by \eqref{e:bass2a} exceeds $KR.$
Notice that
$$\frac12 \gamma|h|^2 \le \langle  \omega A^{-2\alpha}\Pl h , h \rangle  + \frac12 \delta \| h \|^2
\qquad \text{for all }h\in \Pl \cV$$
requires that
\begin{equation}
\label{eq:un}
\frac12 \gamma \le \frac{\omega}{|k|^{4\alpha}}+\frac12 \delta |k|^2
\qquad \text{for all } |k|^2<\lambda L^2/4\pi^2.
\end{equation}
On the other hand,
$$\frac12 \gamma|h|^2 \le \langle  \omega A^{-2\alpha}\Pl h , h \rangle  + \frac12 \delta \| h \|^2
\qquad \text{for all }h\in \Ql \cV$$
requires that
\begin{equation}
\label{eq:deux}
\gamma \le \delta |k|^2 \qquad \text{for all } |k|^2
\ge \lambda L^2/4\pi^2.
\end{equation}
Since the global minimum of the function $x \in \bbR^+
\mapsto \omega x^{-2\alpha}+\frac12 \delta x$ occurs at a point
$c\delta^{2\alpha/2\alpha+1}\omega^{1/2\alpha+1}$
we see that the maximum value of $\gamma$ such that \eqref{e:bass2a} holds, $\gamma_{\rm max}$, is
$$\gamma_{\rm max} =\min\Big\{\frac{\delta \lambda L^2}{4\pi^2},c(\delta^
{2\alpha}\omega)^{1/(2\alpha+1)}\Big\}.$$
%$$\gamma \ge \min\Big\{\frac{\delta \lambda L^2}{4\pi^2},c\delta^{2\alpha/2\alpha+1}\omega^{1/2\al%pha+1}\Big\}.$$
This demonstrates that, provided $\lambda$ is large enough,
and $\omega$ is large enough, then the conditions of the
theorem are satisfied. 

In summary, these conditions are satisfied
provided that enough of the low Fourier modes are
observed ($\lambda$ large enough), 
and provided that the ratio of the scale of the
covariance for the model to that for the observations, $\omega$,
is sufficiently large. Ensuring that the latter is achieved
is often termed {\em variance inflation} in the applied
literature and our theory provides concrete analytical
insight into the mechanisms behind it. 
Furthermore, notice that once
$\lambda$ and $\omega$ are chosen to ensure this, 
then the asymptotic
mean square error will be small, 
provided $\epsilon:=\omega\sigma_0$ is small that is,
provided the observational noise is sufficiently small.
In this situation the theorem establishes
a form of accuracy of the filter since, regardless of the
starting point of the filter, 
\[
\limsup_{t\to\infty} \EX|\hat{m}(t)-u(t)|^2 \leq 
\frac{1}{\gamma_0}\epsilon^2{\trace}_{\cH}\bigl(A^{-4\alpha-2\beta}\Pl\bigr).
\]
\end{rem}

%%%%%%%%%%%%%%%%%%%%%%%%%%%%%%%%%%%%%%%%%%%%%

\subsection{Forward Stability in Probability}
\label{ssec:z}

%%%%%%%%%%%%%%%%%%%%%%%%%%%%%%%%%%%%%%%%%%%%%

The aim of this section is to prove that two
different solutions of the continuous 3DVAR filter
will converge to one another in
probability as $t \to \infty.$
Almost sure and mean square convergence is out of reach
in forward time. However, almost sure pullback convergence 
is possible and we study this in the next section.

Throughout this section we define, for $u$ on the
attractor,
\begin{equation}
\label{eq:useful}
R'=\sup_{t \in \bbR}|f+\omega A^{-2\alpha}\Pl u|_{-1}^2
\end{equation}
and we assume that $R'<\infty.$
From this we define
\begin{equation}
\label{e:defR2}
R''=
\frac{K}{\delta^2}R'+\frac{K}{\delta}\omega^2\sigma_0^2 
{\trace}_{\cH}\bigl(A^{-4\alpha-2\beta}\Pl\bigr). 
\end{equation}

\begin{theo}
\label{t:stabf}
Let $\hat{m}_i$ solve \eqref{eq:nse2} with initial
condition $\hat{m}(0)=\hat{m}_i(0)$ 
and let $u$ solve \eqref{eq:nse} with initial condition
on the global attractor $\cA$.
For $\lambda=\infty$
assume that $4\alpha+2\beta>1$ and $\alpha>-\frac12$. 
Let $R''$ be defined as above, and
suppose $\gamma$, the largest positive number
such that \eqref{e:bass2a} holds, satisfies 
$\gamma = R''+\gamma_0$ for some $\gamma_0>0.$ 

Then for all $\eta\in(0,\gamma_0)$  
\[
 |\hat{m}_1(t)-\hat{m}_2(t)|e^{\eta t} \to 0
\qquad\text{in probability as } t \to \infty.
\]
\end{theo}

\begin{proof} It follows from Lemma \ref{lem:stabf} below that,
for any fixed $t>0$,
\begin{equation}
\label{e:convprob}
 \mathbb{P}\Big(
|\hat{m}_1(t)-\hat{m}_2(t)|^2 
\leq 
|\hat{m}_1(0)-\hat{m}_2(0)|^2 e^{-2\gamma_0 t} 
 \Big)
\geq 
 \mathbb{P}\Big(
\frac{1}{t}\int_0^t K\|\hat{m}_2(s)\|^2 ds \leq \gamma - \gamma_0
 \Big)\;.
\end{equation}
Thus to establish the desired convergence in probability,
it suffices to establish that the right hand side 
converges to $1$ as $t\to\infty$.

Taking the inner-product of equation \eqref{eq:nse2} with
$\hm$, applying \eqref{eq:Best2} and  
using the It\^o formula from Theorem 4.17
of \cite{Dap-Z}, we obtain 
\begin{align*}
\tfrac12 d |\hm |^2 \leq&
\Bigl(-\delta \| \hm \|^2 
-\langle \omega A^{-2\alpha}\Pl \hm , \hm \rangle 
+\langle f+\omega A^{-2\alpha}\Pl u,\hm \rangle\Bigr)dt\\ 
&\qquad \qquad
+ \langle \hm ,\omega \sigma_0 A^{-2\alpha-\beta} \Pl dW\rangle
+ \tfrac12 \omega^2\sigma_0^2 {\trace}_{\cH}\bigl(A^{-4\alpha-2\beta}\Pl 
\bigr) dt\\
\leq&
\Bigl(-\frac12\delta \| \hm \|^2                        
+\frac{1}{2\delta}|f+\omega A^{-2\alpha}\Pl u|^2_{-1}\Bigr)dt\\
&\qquad \qquad
+ \langle \hm ,\omega \sigma_0 A^{-2\alpha-\beta} \Pl dW\rangle
+ \tfrac12 \omega^2\sigma_0^2 {\trace}_{\cH}\bigl(A^{-4\alpha-2\beta}\Pl 
\bigr) dt\;.
\end{align*}
%
%Here we have used the fact that the trace of $A^{s}$ in $\cH$
%dominates the trace in $Y$.
 Notice that, from the Poincar\'e 
inequality, we have that
\begin{equation}
\label{eq:L}
d |\hm |^2 \leq
\Bigl(-\delta | \hm |^2 
+\frac{1}{\delta}R'
+ \omega^2\sigma_0^2 {\trace}_{\cH}\bigl(A^{-4\alpha-2\beta}\Pl 
\bigr)
\Bigr)dt
+ 2\langle \hm ,\omega \sigma_0 A^{-2\alpha-\beta} \Pl dW\rangle
\;.
\end{equation}
From this inequality we can deduce two facts. 
Taking expectations gives
\[
d\bigl(\E|\hm(t) |^2\bigr) \leq
\Bigl(-\delta \E|\hm|^2 
+\frac{1}{\delta}R'\Bigr)dt
+ \omega^2\sigma_0^2 {\trace}_{\cH}\bigl(A^{-4\alpha-2\beta}\Pl 
\bigr)dt;,
\]
and thus 
with $R''$ from (\ref{e:defR2})
\begin{equation}
\label{eq:first}
\limsup_{t \to \infty}\E|\hm(t)|^2 \le 
\frac{1}{\delta^2}R'+
\frac{1}{\delta}\omega^2\sigma_0^2 {\trace}_{\cH}\bigl(A^{-4\alpha-2\beta}\Pl\bigr)
=\frac{R''}{K}.
\end{equation}
We also see that
\begin{equation}\label{eq:second}
\frac{1}{t}\int_0^t |\hm(s)|^2 ds \le 
\frac{R''}{K}
+\frac{1}{t}|\hm(0)|^2+ I(t)\;,
\end{equation}
where we have defined
\[
 I(t)=\frac{2}{t}\int_0^t
\bigl\langle \hm(s) ,\omega \sigma_0 A^{-2\alpha-\beta}\Pl dW(s)
\bigr\rangle
\]
Observe that, by the It\^o formula,
\[
\E |I(t)|^2 
\le \frac{c}{t^2}\int_0^t \E|\hm(s)|^2 ds
\]
for the positive constant 
$c=\omega^2 \sigma_0^2 \| A^{-2\alpha-\beta}\Pl\|^2_{\mathcal{L}(\cH)}$.
Using \eqref{eq:first} we deduce that 
$I(t) \to 0$ in mean square and hence in probability.
As a consequence we deduce that \eqref{eq:second} implies
that 
\[
\mathbb{P}\Big(
  \frac1t  \int_0^t K\|\hu(s)\|^2 ds
\leq  R'' \Big) \to 1 
\quad\text{for }t\to\infty\;. 
\]
This completes the proof.
\end{proof}

\begin{lem}
\label{lem:stabf}
Let $\hat{m}_i$ solve \eqref{eq:nse2} 
with the same $u$ on the attractor but with different 
initial
conditions $\hat{m}_i(0)$. 
For $\lambda=\infty$
assume that $4\alpha+2\beta>1$. 
Fix $t>0$
and recall that
$K$ is the constant appearing in \eqref{eq:Best}.
Suppose that $\gamma$, the largest positive number
such that \eqref{e:bass2a} holds, satisfies 
$$
\frac{1}{t}\int_0^t K\|\hat{m}_2(s)\|^2 ds+\gamma_0 \leq \gamma
$$
for some $\gamma_0>0.$ 
Then 
\[
|\hat{m}_1(t)-\hat{m}_2(t)|^2 
\leq e^{-2\gamma_0 t}|\hat{m}_1(0)-\hat{m}_2(0)|^2.
\]
\end{lem}

\begin{proof}
We define the error $e=\hat{m}_1-\hat{m}_2$,
subtract equation \eqref{eq:nse2} from itself and take
the inner-product with $e$ to obtain, using 
\eqref{eq:Best}, 
\begin{eqnarray*}
\tfrac12 \frac{d }{dt} |e |^2
&=& \langle \cf(\hat{m}_1)-\cf(\hat{m}_2) , e \rangle 
-\langle  \omega A^{-2\alpha}\Pl e , e \rangle  
\\&\leq&
-\delta \| e \|^2 + K \| \hat{m}_2(t) \|^2  | e |^2 
-\langle  \omega A^{-2\alpha}\Pl e , e \rangle\;.
\end{eqnarray*}
Applying \eqref{e:bass2a} we obtain
\[
 \tfrac12 \frac{d }{dt} |e (t)|^2
\leq ( K \| \hat{m}_2(t) \|^2 -\gamma ) \cdot | e(t) |^2 \;.
\]
Integrating this inequality yields 
\[
|e(t) |^2 
\leq   \exp\Bigl(2\int_0^t( K \| \hat{m}_2(t) \|^2 -\gamma )ds\Bigr)\cdot  | e(0) |^2 
\;.\]
\end{proof}
This gives the desired result.

\begin{rem}
Satisfying the condition on $\gamma$ for the stability 
Theorem \ref{t:stabf} is harder than for the
accuracy Theorem \ref{thm:mainacc}. This is because
$R''$ can grow with $\omega$ and so analogous arguments
to those used at the end of the previous subsection
may fail. However different proofs can be developed, in the
case where $\sigma_0$ is sufficiently small, to overcome
this effect.
\end{rem}

%%%%%%%%%%%%%%%%%%%%%%%%%%%%%%%%%%%%%%%%%

\section{Pullback Accuracy and Stability}
\label{sec:PB}

%%%%%%%%%%%%%%%%%%%%%%%%%%%%%%%%%%%%%%%%%%%%%

In this section we consider almost 
sure accuracy and stability 
results for the 3DVAR algorithm 
applied to the 2D-Navier-Stokes equation.
We use the notion of pullback convergence as pioneered
in the theory of stochastic dynamical systems, as 
we do not expect almost sure results to hold forward in time.
Indeed, as shown in the previous section,
convergence in probability is typically the
result of forward studies of stability.

The methodology that we employ derives from the study
of semilinear equations driven by additive noise;
in particular the Ornstein-Uhlenbeck (OU)
process constructed from a modification of the
Stokes' equation plays a central role. 
Properties of this process are
described in subsection \ref{ssec:one}, and the
necessary properties of the 3DVAR Navier-Stokes filter
are discussed in subsection \ref{ssec:two}. In both
subsections a key aspect of the analysis concerns
the extension of solutions to the whole real line
$t \in \bbR$. Subsections \ref{ssec:three}
and \ref{ssec:four} then concern accuracy and stability
for the filter, in the pull-back sense.

In the following we define the Wiener process
$$
\WP:=  \omega\sigma_0 A^{-2\alpha-\beta}\Pl W,
$$ 
and recall that when $\lambda=\infty$ 
we assume $4\alpha+2\beta>1$ and 
$\alpha>-\frac12$. In this
section the driving Brownian motion is considered 
to be two-sided: $\WP \in C(\bbR,\cH)$. This
enables us to study notions of pullback attraction and
stability. 
With this definition, 3DVAR  for (\ref{eq:nse}), namely
equation \eqref{eq:nse2}), may be written
\begin{equation}
\label{eq:M}
\frac{d\hat{m}}{dt} 
+ \delta A\hat{m} + \cB(\hat{m}, \hat{m})  + \omega A^{-2\alpha} \Pl(\hat{m}-u)
= f+  \frac{d\WP}{dt}, 
\quad \hat{m}(0)=\hat{m}_0\;.
\end{equation}

We employ the same notations from the previous sections 
for the nonlinearity $\cf(u)$, the Stokes operator $A$, 
the bilinear form $\cB$, and the 
spaces $\cH$ and $\cV.$ 

%

%%%%%%%%%%%%%%%%%%%%%%%%%%%%%%%%%%%%%%%%%%

\subsection{Stationary Ornstein-Uhlenbeck Processes}
\label{ssec:one}

%%%%%%%%%%%%%%%%%%%%%%%%%%%%%%%%%%%%%%%%%%%

Let $\phi\geq 0$ and define
the  stationary ergodic OU process $Z_\phi$ as follows,
using integration by parts to find the second expression:
\begin{subequations}
\label{e:defZ}
\begin{align}
Z_\phi(t) &:= \int_{-\infty}^t  e^{-(t-s)(\delta A+\phi)}d\WP(s)\\
&= \WP(t) - \int_{-\infty}^t (\delta A+\phi) e^{-(t-s)(\delta A+\phi)}\WP(s) ds  \;.
\end{align}
\end{subequations}
Note that $Z_\phi$ satisfies
\begin{equation}
\label{eq:ZEQ}
\partial_t Z_\phi 
+  (\delta A+\phi) Z_\phi 
= \partial_t \WP.
\end{equation}
With a slight abuse of notation we 
rewrite the random variable 
$Z_\phi(0)$ as $Z_\phi(\WP)$, a function of the whole Wiener path $t\mapsto \WP(t)$.
Thus
$Z_\phi(t)=Z_\phi(\theta_t\WP)$, where $\theta_t$ is the stationary ergodic shift on 
Wiener space defined by 
\[
 \theta_t\WP(s) = \WP(t+s)-\WP(t) 
\qquad\text{for all } t,s\in\R.
\]
The noise is always of trace-class, in case either $\lambda<\infty$ or  $4\alpha+2\beta>1$.
Recall that by Lemma \ref{lem:SC}, 
the OU-process $Z_\phi$ has a version 
with continuous paths in $\cV$. 
We will always assume this in the following.
It is well known that $Z_\phi$ 
satisfies the Birkhoff ergodic theorem,
because it is a stationary ergodic process; 
we now formulate this fact in the pullback sense.

%%%%%%%%%%%%%%%%%%%%%%%%%%%%%%%%%%%%%%%%%%%%%%%%%%

\begin{theo}[\bf{}Birkhoff Ergodic Theorem]
For $\lambda=\infty$ 
assume that $4\alpha+2\beta>1$. 
Then
\[
\limsup_{s\to \infty} \frac1{s}\int_{-s}^0 \|Z_\phi(\tau)\|^2 d\tau 
= \EX \|Z_\phi(0)\|^2\;.
\]
\end{theo}

%%%%%%%%%%%%%%%%%%%%%%%%%%%%%%%%%%%%%%%%%%%%%%

\begin{proof}
Just note that $Z_\phi(\tau)= Z_\phi(\theta_\tau\WP)$,
and thus 
\[\frac1{s}\int_{-s}^0 \|Z_\phi(\tau)\|^2 d\tau =  
\frac1{s}\int_0^{s} \|Z_\phi(\theta_{-\tau}\WP)\|^2 d\tau
\to \E \|Z_\phi(\WP)\|^2
\quad \text{for } s\to\infty
\]
by 
the classical version of the
Birkhoff ergodic theorem, 
as $\theta_{-\tau}\WP$, $\tau\geq0$
is stationary and ergodic.
\end{proof}

%%%%%%%%%%%%%%%%%%%%%%%%%%%%%%%%%%%%%%%%%%%%%%%%%

We can reformulate the implications of the 
ergodic theorem in several ways.

\begin{cor}
For $\lambda=\infty$ 
assume that $4\alpha+2\beta>1$. 
There exists a random constant $C(\WP)$ such that 
\[
 \frac1{|s|}\int_s^0 \|Z_\phi(\tau)\|^2 d\tau 
\leq  
C(\WP) 
\qquad \text{ for all }s<0.
\]
Furthermore, for any $\epsilon>0$ there is a random time 
$t_\epsilon(\WP)<0$ such that 
\[
  \frac1{|s|}\int_s^0 \|Z_\phi(\tau)\|^2 d\tau 
\leq   
(1+\epsilon) \EX \|Z_\phi(0)\|^2
\qquad \text{ for all }s<t_\epsilon(\WP)<0.
\]
\end{cor}
%
%%%%%%%%%%%%%%%%%%%%%%%%%%%%%%%%%%%%%%%%%%%%
%
This result immediately implies 
$$
\frac1{t-s}\int_s^t \|Z_\phi(\theta_\tau\WP)\|^2 d\tau 
=
\frac1{t-s}\int_{s-t}^0 \|Z_\phi(\theta_{\tau+t}\WP)\|^2 d\tau 
\leq C(\theta_t\WP)\;.
$$
Finally we observe that it is well-known that
the Ornstein-Uhlenbeck process $Z_\phi$ is 
a tempered random variable,
which means that $Z_\phi(\theta_s\WP)$ grows sub-exponentially for $s\to -\infty$,
and in fact it grows slower that any polynomial.
We now state this precisely.
%
%%%%%%%%%%%%%%%%%%%%%%%%%%%%%%%%%%%%%%%%%%%%%%%%%%
\begin{lem}
\label{lem:SCsub}
For $\lambda=\infty$ 
assume that $4\alpha+2\beta>1$. 
Then on a set of measure one
\[
\lim_{s\to -\infty}\|Z_\phi(s)\| \cdot |s|^{-\epsilon} =0
\qquad \text{for all } \epsilon >0\;.
\]
\end{lem}
\begin{proof}
The claim follows from Proposition 4.1.3 of \cite{Arn98} 
which states that for any positive functional $h$ on Wiener paths
such that $ \mathbb{E} \sup_{t\in[0,1]}h(\theta_t\WP) <\infty $ 
one has $\lim_{t\to\infty} \frac1t h(\theta_t\WP) =0$.
Here $h(\WP)= \|Z_\phi(\WP)\|^p$, where the moment is finite 
due to Lemma \ref{lem:SC}. 
\end{proof}

In addition to the preceding almost sure result,
the following moment bound on $Z_\phi$ is also
useful. It shows that $Z_\phi$ is of order $\sigma_0$ 
and converges to $0$ for $\phi\to\infty$.
In the following it may be useful to play with $\phi$, 
and even to use random $\phi$,
as our estimates hold path-wise for all $\phi$.
%
%%%%%%%%%%%%%%%%%%%%%%%%%%%%%%%%%%%%%%%%%%%%%%%%%%%%

%%%%%%%%%%%%%%%%%%%%%%%%%%%%%%%%%%%%%%%%%%%%%%%%%%%

\begin{lem}
\label{sc:size}
For $\lambda=\infty$ 
assume that $4\alpha+2\beta>1$. 
Then, for all $p>1$ there is a constant $C_p>0$
such that 
\[
\Big(\EX\|Z_\phi(t)\|^{2p} \Big)^{1/p}  
\leq 
C_p \omega^2\sigma_0^2 \cdot
\trace\{ (\delta A+\phi)^{-1}A^{1-4\alpha-2\beta}\Pl \},
\qquad\forall\;t\in\R.
\]
\end{lem}
%
%%%%%%%%%%%%%%%%%%%%%%%%%%%%%%%%%%%%%%%%%%%%%%%%%%%%
%
\begin{proof}
 Due to stationarity it is sufficient to consider 
$\EX\|Z_\phi(0)\|^{2p}$. Due to Gaussianity it is enough 
to consider $p=1$.
\[\EX\|Z_\phi(0)\|^2
=\EX|A^{1/2} Z_\phi(0)|^2 
=   \omega^2\sigma_0^2 \EX\Big|A^{1/2}\int_{-\infty}^0 e^{s(\delta A+\phi)} A^{-2\alpha-\beta}\Pl dW(s)\Big|^2 
\;.\]
Thus by the It\^o-Isometry
we obtain 
(projection $\Pl$ commutes with $A$)
\begin{eqnarray*}
 \EX\|Z_\phi(0)\|^2
&=& \omega^2\sigma_0^2\cdot \trace\Bigl(\int_{-\infty}^0  e^{2s(\delta A+\phi)} A^{1-4\alpha-2\beta}\Pl ds \Bigr) \\
&= &  \frac12 \omega^2\sigma_0^2\cdot \trace
\Bigl( (\delta A+\phi)^{-1}A^{1-4\alpha-2\beta}\Pl\Bigr)\;.
\end{eqnarray*}
\end{proof}

\begin{rem}
\label{rem:eps2}
A key conclusion of the preceding lemma is
that, if $\epsilon:=\omega\sigma_0$ (as defined in
Remark \ref{rem:eps}) is small, then all moments of 
the OU process $Z_{\phi}$ are small. Furthermore, 
the parameter $\phi$ can be tuned to make these moments
as small as desired.
\end{rem}
%%%%%%%%%%%%%%%%%%%%%%%%%%%%%%%%%%%

\subsection{Solutions Continuous Time 2D Navier-Stokes Filter}
\label{ssec:two}

%%%%%%%%%%%%%%%%%%%%%%%%%%%%%%%%%%%%%%%%

In the following we denote the solution of (\ref{eq:M}) 
with initial condition $\hat{m}(s)=\hat{m}_0$ and given Wiener path $\WP$
by $S(t,s,\WP)\hat{m}_0$.
This object forms  
a stochastic dynamical system (SDS); 
see \cite{CF94, CrDebuFl}.
We cannot use directly the notion of a random dynamical system,
as in \cite{Arn98}, because of the non-autonomous 
forcing $u$ in \eqref{eq:M}.

The fact that the solution of the SPDE (\ref{eq:M}) can be defined 
path-wise for every fixed path of $\WP$,  can be seen from the 
well-known method of changing to the variable 
$v:=\hat{m}-Z_\phi$. 
(see Section 7 of \cite{CF94} or Chapter 15 of \cite{dap2},
for example). Now, since $Z_\phi$ satisfies (\ref{eq:ZEQ}),
subtraction from (\ref{eq:M}) shows that
$v$ solves the random PDE
\begin{equation}
\label{e:randPDE}
 \frac{d}{dt} v 
+ \delta A v + \cB(v,v) +2\cB(v,Z_\phi)+\cB(Z_\phi,Z_\phi)  
+  \omega A^{-2\alpha} (v+Z_{\phi}-u) - \phi Z_\phi
= f\;.
\end{equation}

This can be solved for each given path of $\WP$ 
with methods similar to the ones used for Proposition 
\ref{prop:1}  
(see also Proposition \ref{prop:ex-S2DNS}).
Once, the solution is  defined path-wise, the generation
of a stochastic dynamical system is straightforward.
Let us summarize this in a theorem:

\begin{theo}[\bf{Solutions}] 
\label{thm:solNS} 
For all $u_0$ on the attractor $\cA$
the Navier-Stokes equation (\ref{eq:nse}) has a solution 
$u\in L^\infty(\bbR,\cV).$ 
Now consider the 3DVAR filter written in the form
of equation (\ref{e:randPDE}).
In the case $\lambda=\infty$ assume 
that $\alpha>-\tfrac12$ and $4\alpha+2\beta>1$. 
For any $s\in\R$, any path of the Wiener process $\WP$,
and any initial condition $v(s)=\hat{m}(s)-Z_\phi(s) \in \cH$ 
equation  (\ref{e:randPDE})
has a unique solution 
\[
v\in C^0_\loc([s,\infty),\cH) \cap L^2_\loc([s,\infty),\cV)\;.
\]
This implies the existence of a stochastic dynamical system 
$S$ for (\ref{eq:M}).
\end{theo}

\begin{proof}
The first statement 
follows directly from Proposition
\ref{prop:1} if we take a solution on the attractor;
in that case it follows that, in fact,
$u\in L^\infty(\bbR,\cV)$.
Proof of the second statement is discussed prior to
the theorem statement.
\end{proof}

%%%%%%%%%%%%%%%%%%%%%%%%%%%%%%%%%%%%%%%%%55

\subsection{Pullback Accuracy}
\label{ssec:three}

%%%%%%%%%%%%%%%%%%%%%%%%%%%%%%%%%%%%%%%%%%%%

Here we show that in the pullback sense 
solutions $\hat{m}$ for large times stay close 
to  $u$, where the error scales with the observational 
noise strength $\sigma_0$.
Recall $K$ and $K'$ defined in Lemma \ref{lem:B1} 
and $R$ the uniform bound on $u$ from Proposition
\ref{prop:1}.

\begin{theo}[\bf{}Pullback Accuracy]
\label{thm:pbacc}
Let $\hat{m}$ solve \eqref{eq:nse2}, and let 
$u$ solve \eqref{eq:nse} with initial condition
on the global attractor $\cA$.
In the case $\lambda=\infty$ assume additionally
that $4\alpha+2\beta>1$ and $\alpha>-\frac12$. 
Suppose that $\gamma$ 
from (\ref{e:bass2a})
is sufficiently large so that
\begin{equation}
 \label{cond:PBA}
 K(17 \EX \|Z_\phi\|^2+ 16 R) < \gamma \;.
\end{equation}
Then  there is a random constant $r(\WP)>0$ 
such that for any initial condition $\hat{m}_0$
\[
  \limsup_{s\to-\infty} | S(t,s,\WP)\hat{m}_0 - u(t)- Z_\phi(\theta_t \WP)|^2
 \leq r(\theta_t \WP) \;.
\]
with a finite constant
\[
r(\WP) = \frac4\delta   
\int_{-\infty}^0 \exp\Bigl( 
\int_\tau^0 \bigl( 16 K (\|Z_\phi\|^2+ R ) - \gamma \bigr) d\eta\Bigr) 
\mathcal{T}^2 d\tau\;,
\]
where $\mathcal{T}:=K'\|Z_\phi\|(\|Z_\phi\|+2\|u\|)+\phi|Z_\phi|+\omega|A^{-2\alpha}\Pl Z_\phi|$
where $K'$ and $K:=(K')^2/\delta$ are as defined
in Lemma \ref{lem:B1}.

\end{theo}

\begin{rem}
Regarding Theorem \ref{thm:pbacc} 
we make the following obervations:

\begin{itemize}

\item In (\ref{cond:PBA}) the contribution
$\EX \|Z_\phi\|^2$ can be made arbitrarily small by choosing
$\phi$ sufficiently large,  
or is small if $\epsilon:=\omega\sigma_0$ is sufficiently small; see Lemma \ref{sc:size}.
Thus  $16 R K < \gamma$ is sufficient for accuracy.
With a more careful application of Young's inequality, we could also get rid of several factors of $2$,
recovering the condition $RK<\gamma$
from the forward accuracy result of Theorem \ref{thm:mainacc}.

\item The assumption that $u$ lies 
on the attractor could we weakened to a 
condition on the limsup of $u$.
We state the stronger condition for 
simplicity of proofs.

\item In the language of random dynamical systems, 
the theorem states that the stochastic dynamical 
system $S(t,s\WP)\hat{m}_0 $
has is a random pullback absorbing 
ball centered around $u(t)$ with radius scaling with the size of 
the stochastic convolution  $Z_\phi$. 
By Lemma \ref{sc:size}, this scales as $\mathcal{O}(\epsilon)$
for $\epsilon=\omega \sigma_0$ sufficiently small; thus
we have derived an accuracy result, in the pullback sense.

\end{itemize}
\end{rem}

\begin{proof}
{\em (Theorem \ref{thm:pbacc})}.
Consider the difference $d= \hat{m}-u$, where $\hat{m}(t)=S(t,s)\hat{m}_0$. This solves 
\begin{equation}
\label{e:defd}
\partial_t d 
+  \delta A d + \cB(d,d) + 2\cB(u,d) +  \omega A^{-2\alpha} \Pl d 
= \partial_t \WP\;.
\end{equation}
In order to get rid of the noise, define 
$\psi = d - Z_\phi = \hat{m} - u - Z_\phi$,
where $Z_\phi$ is the stationary stochastic convolution.
Since $Z_\phi$ solves (\ref{eq:ZEQ}) 
the process $\psi$ solves 
\begin{equation}
\label{e:defpsi}
\partial_t \psi 
+ \delta A \psi + \cB(\psi+Z_\phi,\psi+Z_\phi) + 2\cB(u,\psi+Z_\phi) 
+  \omega A^{-2\alpha}  \Pl (\psi+Z_\phi) -\phi Z_\phi = 0\;.
\end{equation}
From this random PDE, 
we can take the scalar product with $\psi$ to obtain, using \eqref{eq:Best2},
\[
\begin{split}
\tfrac12 \partial_t |\psi|^2 
+ \delta \|\psi\|^2 
& =
-\langle 2\cB(Z_\phi,\psi)+ \cB(Z_\phi,Z_\phi) + 2\cB(u,\psi+Z_\phi) ,\psi\rangle 
\\& \qquad 
- \langle \omega A^{-2\alpha} \Pl (\psi+Z_\phi) -\phi Z_\phi ,\psi\rangle
 \;.
\end{split}
\]
Using (\ref{e:bass2a}) and Lemma \ref{lem:B1} we obtain
\[
\begin{split}
 \tfrac12 \partial_t |\psi|^2 + \tfrac\delta2 \|\psi\|^2 +  \tfrac\gamma2 |\psi|^2
& \leq -\langle 2\cB(Z_\phi,\psi)+ \cB(Z_\phi,Z_\phi) + 2\cB(u,\psi+Z_\phi) ,\psi\rangle 
\\& \qquad 
- \langle \omega A^{-2\alpha} \Pl Z_\phi -\phi Z_\phi ,\psi\rangle
\\ & \leq
2 K' \left(\|Z_\phi\| + \|u\|\right) \cdot\|\psi\|  \cdot|\psi|\\
&\qquad 
+ K' \|Z_\phi\|\cdot  \left(\|Z_\phi\|+2\|u\|\right) \cdot \|\psi\|\\
&\qquad + \left(\phi|Z_\phi|+ \omega |A^{-2\alpha}\Pl Z_\phi| \right) \cdot|\psi|
\;.
\end{split}
\]
Recall that
\[
\mathcal{T}= K' \|Z_\phi\| (\|Z_\phi\|+2\|u\|)
+\phi|Z_\phi|+ \omega |A^{-2\alpha}\Pl Z_\phi|\;.
\]
Thus we have, 
using the Young inequality in the form 
$ab \leq \frac1\delta a^2 + \frac\delta4  b^2$ 
twice,
\[
\begin{split}
 \tfrac12 \partial_t |\psi|^2 + \tfrac\delta2 \|\psi\|^2 +  \tfrac\gamma2 |\psi|^2
& 
\leq 
2 K' (\|Z_\phi\| + \|u\|)\|\psi\| |\psi|
+  \mathcal{T}\cdot \|\psi\|\\
&\leq 
4 K (\|Z_\phi\| + \|u\|)^2 |\psi|^2
+  \tfrac1\delta\mathcal{T}^2 + \tfrac\delta2\|\psi\|^2,
\end{split}
\]
since $K=(K')^2/\delta.$ 
Hence
\[
\partial_t |\psi|^2 + \gamma |\psi|^2
\leq
8K( \|Z_\phi\| + \|u\|)^2 \cdot |\psi|^2 
+  \tfrac2\delta \mathcal{T}^2 \;.
\]

Comparison principle with $\psi(s)=\hat{m}_0-u(s)-Z_\phi(s)$ yields 
(using the bound on $u$ and $(a+b)^2 \leq 2a^2+2b^2$)
\begin{eqnarray*}
\lefteqn{
| \psi(t) |^2 \leq 
|\hat{m}_0-u(s)-Z_\phi(s) |^2  
\exp\Big( \int_s^t [ 16 K ( \|Z_\phi\|^2 + R) - \gamma] dr\Big)
}\\
&& + \frac2\delta  \int_s^t 
\exp\Big( \int_r^t [16K ( \|Z_\phi\|^2+R) 
-\gamma] d\tau\Big) \mathcal{T}^2 dr
\end{eqnarray*}
Thus, we can now use Birkhoffs theorem and the sub-exponential growth 
for $Z_\phi$. We obtain for $\gamma$ sufficiently large 
(as asserted by the Theorem), 
that there is a random time 
$t_0(\WP)<0$  such that for all  $s< t_0(\WP)<0$
\[
 | \psi(t) |^2 
\leq \frac4\delta  \int_s^t 
\exp\Big( \int_r^t [16K ( \|Z_\phi\|^2+R) 
-\gamma] d\tau\Big) \mathcal{T}^2 dr\;.
\]
Recall that $\psi(t) =S(t,s,\WP)\hat{m}_0-Z(t)-u(t)$.
This finishes the proof, as the right hand side  
is almost surely a finite random constant,
due to Birkhoffs ergodic theorem and sub-exponential growth 
of $Z_\phi$ and hence $\mathcal{T}^2$ (see Lemma \ref{lem:SCsub}).
\end{proof}

%%%%%%%%%%%%%%%%%%%%%%%%%%%%%%%%%%%%%%%%%55

\subsection{Pullback Stability}
\label{ssec:four}

%%%%%%%%%%%%%%%%%%%%%%%%%%%%%%%%%%%%%%%%%%%%

Now we verify that under suitable conditions
all solutions of (\ref{eq:M}) pullback converge exponentially fast 
towards each other.
We make the assumption that Birkhoff bounds 
hold for the solution  $\hat{m}$ (see theorem statement
below to make this assumption precise). These bounds
do not follow directly from Birkhoffs ergodic theorem,
as the equation is non-autonomous due to the presence of $u$.
Whilst it should be possible to establish such bounds,
using the techniques in \cite{FlGa:95} or \cite{ESSt:10}, 
doing so is technically involved, as one needs to use 
random $\phi$'s in the definition of $Z_\phi$.
In order to keep the presentation at a reasonable level,
we refrain from giving details on this point.

\begin{theo}[Exponential Stability]
\label{thm:huexpstab}
Assume there is one initial condition $\hat{m}_0^{(1)}$ 
such that the corresponding solution $S(t,s,\WP)\hat{m}_0^{(1)} $
satisfies Birkhoff-bounds.
To be more precise, we assume that $\gamma$ 
from (\ref{e:bass2a}) is sufficiently large that, 
for some $\eta>0$ and $K=(K')^2/\delta$ from Lemma \ref{lem:B1}, 
\begin{equation}
\label{e:Birkbou}
\limsup_{s\to-\infty} \frac{4K}{t-s} \int_s^t  \| S(\tau,s,\WP)\hat{m}_0^{(1)}\|^2 d\tau 
< \gamma -2\eta\;. 
\end{equation}
Let $\hat{m}_0^{(2)}$ be 
any other initial condition.
Then
\[
\lim_{s\to-\infty}| S(t,s,\WP)\hat{m}_0^{(1)} -S(t,s,\WP)\hat{m}_0^{(2)} | 
\cdot \mathrm{e}^{\eta(t-s)} 
=0.
\]
\end{theo}
Recall we verified in Theorem \ref{thm:pbacc}
that (\ref{eq:M}) 
has a random pullback absorbing set in $L^2$ centered around $u(t)$. 
Together with Theorem \ref{thm:huexpstab}
this immediately implies that 
Equation (\ref{eq:M}) 
has a random pullback attractor in $L^2$ consisting of a single point
that attracts all solutions. 
Let us remark, that we did not show that the attractor also pullback-attracts 
tempered bounded sets,
but this is a straightforward modification.

\begin{proof}{\em (Theorem \ref{thm:huexpstab})}
Define here $v=\hat{m}_1-\hat{m}_2$, where $\hat{m}_i(t)=S(t,s,\WP)\hat{m}_0^{(i)}$ 
are solutions of (\ref{eq:M}) with different initial conditions.
It is easy to see by the symmetry of $\cB$
that 
\[
\partial_t v +  \delta Av + \cB( \hat{m}_1 + \hat{m}_2 , v) +  \omega A^{-2\alpha}  \Pl v = 0
\]
or
\[
\partial_t v +  \delta Av + 2\cB( \hat{m}_1, v) - \cB(v,v) + \omega A^{-2\alpha}  \Pl v = 0\;.
\]
Thus
\[
\frac12 \partial_t |v|^2 + \delta\|v\|^2 
+ \omega \langle A^{-2\alpha} \Pl v,v\rangle 
\le 2K' \|\hat{m}_1\| \|v\| |v|\;.
\]
By (\ref{e:bass2a})
\[
 \partial_t |v|^2 + \delta\|v\|^2 +\gamma |v|^2 
\le 4K' \|\hat{m}_1\| \|v\| |v| \;.
\]
Hence, using Young's inequality
($ab \leq \frac1{4\delta}a^2+\delta b^2$) 
with $K=(K')^2/\delta$

\[
 \partial_t |v(t)|^2 +  \gamma |v|^2 \leq    4K \|\hat{m}_1\|^2 |v|^2\;.
\]
Thus, using the comparison principle,
\[
|v(t)|^2 \leq |v(s)|^2 \exp\bigl( \int_s^t [ 4K \|\hat{m}_1\|^2 - \gamma] dr \bigr).  
\]
This converges to $0$ exponentially fast,
provided $\gamma$ is sufficiently large, 
as $v(s)= \hat{m}_0^{(1)}-\hat{m}_0^{(1)}$.
Moreover,
\[
|v(t)|^2 \mathrm{e}^{2\eta (t-s)} 
\leq |v(s)|^2 \exp\bigl( (t-s) r(t,s) \bigr).  
\]
with 
$r(t,s)= \frac{1}{t-s} \int_s^t [ 4K \|\hat{m}_1\|^2d\tau -  \gamma +2\eta $
and 
$\limsup_{s\to-\infty} r(t,s) < 0$ by assumption.
This implies the claim of the theorem.
\end{proof}

\section{Numerical Results}
\label{sec:num}

In this section we study the SPDE \eqref{eq:nse2}
by means of numerical experiments, illustrating 
the results of the previous sections. 
We invoke a split-step scheme to solve equation (\ref{eq:nse2}),
in which we compose numerical integration of the Navier-Stokes
equation (\ref{eq:nse}) 
with numerical solution of the Ornstein-Uhlenbeck process
\begin{equation}
\frac{d \hat{m}}{d t} + 
\omega A^{-2\alpha}(\hat{m}-u)= \omega \sigma_0
A^{-2\alpha-\beta}\frac{dW}{dt}, 
\quad \hat{m}(0)=\hat{m}_0,
\label{eq:nse222}
\end{equation}
at each step. The Navier-Stokes equation (\ref{eq:nse})
itself is solved by a pseudo-spectral method based on
the Fourier basis defined through \eqref{eq:fb}, whilst
the Ornstein-Uhlenbeck process is approximated by
the Euler-Maruyama scheme \cite{kloeden1994stochastic}.
All the examples concern the case $\lambda=\infty$ only;
however similar results are obtained for finite, but sufficiently
large, $\lambda.$

%\begin{itemize}

% \item We should divide into sections illustrating:
% Forward Accuracy (this is the current subsection 6.1);
% Forward Stability (this is Figures 7--10 in subsection 6.2,
% and also concerns forward accuracy but no need to make this
% explicit); Pullback Accuracy and Stability (this is 
% just Figure 12 but should we have more?).

% \item Each subsection should explicitly state which theorem(s)
% we are illustrating.

%\item Figure captions need to be clearer. Amongst other things:
%(i) $r<1$ is mentioned in Figures 5/6 yet that parameter
%belongs to another paper and not this one; 
% (ii) $\ell^2$
% norm is mentioned in Figure 1, it should be $L^2$; also
% $L^2$ norm should be mentioned in all other captions to
% be consistent (or first figure referred to for details
% when format identical just different parameters); 
% (iii) Figures
% Figures 7,8,9 yes good to link to earliers figures, but
% more information needed on top of this. 

% \item Should we remove the Figure 11 as, although it is
% interesting to me, it is a distraction from our main themes?
% Or do we keep it for comparison with Figure 10? In which
% case Figure 10/Figure 11 comparison needs more discussion.
% Also captions need work.

% \item Figure 9 caption. The discussion about random attractor
% not being a point does not really belong in the caption.
% In any case we have not explicitly worked with random attractors
% in the theorems. So include disucssion in-text, not in-caption,
% and make it in terms linked to the theorems proved.

%\end{itemize}

\subsection{Forward Accuracy}
\label{sec:num_forward}

In this section, we will illustrate the results of Theorem \ref{thm:mainacc}.
We will let $\alpha=1/2$ throughout; since $\beta$ is always
non-negative the trace-class noise condition $4\alpha+2\beta>1$
is always satisfied.
Notice that the parameter $\omega$
sets a time-scale for relaxation towards the true signal,
and  $\sigma_0$ sets a scale for the size of fluctuations
about the true signal. The parameter $\beta$
rescales the fluctuation size in the observational
noise at different wavevectors
with respect to the relaxation time.
First we consider setting $\beta=0.$
In Fig. \ref{a1c5.t} we show numerical experiments 
with $\omega=100$ and $\sigma_0=0.05.$ 
We see that the noise level
on top of the signal in the low modes is almost $O(1)$, and
that the high modes do not synchronize at all; the total error 
remains $O(1)$ although trends in the signal are followed.  
On the other hand, for the smaller value of $\sigma_0=0.005$,
still with $\omega=100$,
the noise level on the signal in the low modes is moderate,
the high modes synchronize sufficiently well, 
and the total error is small; this is shown in Fig. \ref{a1c05.t}.

Now we consider the case $\beta=1$. 
Again we take $\omega=100$ and
$\sigma_0=0.05$ and $0.005$
in Figures \ref{a1dc5.t}
and \ref{a1dc05.t}, respectively.  
The synchronization is stronger than
that observed for $\beta=0$ in each case. This is because 
the forcing noise decays more rapidly for large wavevectors when
$\beta$ is increased, as can be observed in the relatively 
smooth trajectories of the high modes of the estimator.

For the case when $\sigma_0=0$ we recover a (non-stochastic)
PDE for the estimator $\hm$.  The values of
$\sigma_0$ and $\beta$ are irrelevant.  The value of $\omega$
is the critical parameter in this case. 
For values of $\omega$ of $O(100)$ the 
convergence is exponentially
fast to machine precision.  For values of $\omega$ of 
$O(1)$ the estimator does not exhibit stable
behaviour.  For intermediate values, 
the estimator may approach the signal and remain bounded and 
still an $O(1)$ distance away (see the case $\omega=10$ 
in Fig. \ref{d10.t}), or else 
it may come close to synchronizing (see the case $\omega=30$
in Fig. \ref{d30.t}).

\begin{figure*} 
\includegraphics[width=1\textwidth, height=8.5cm]{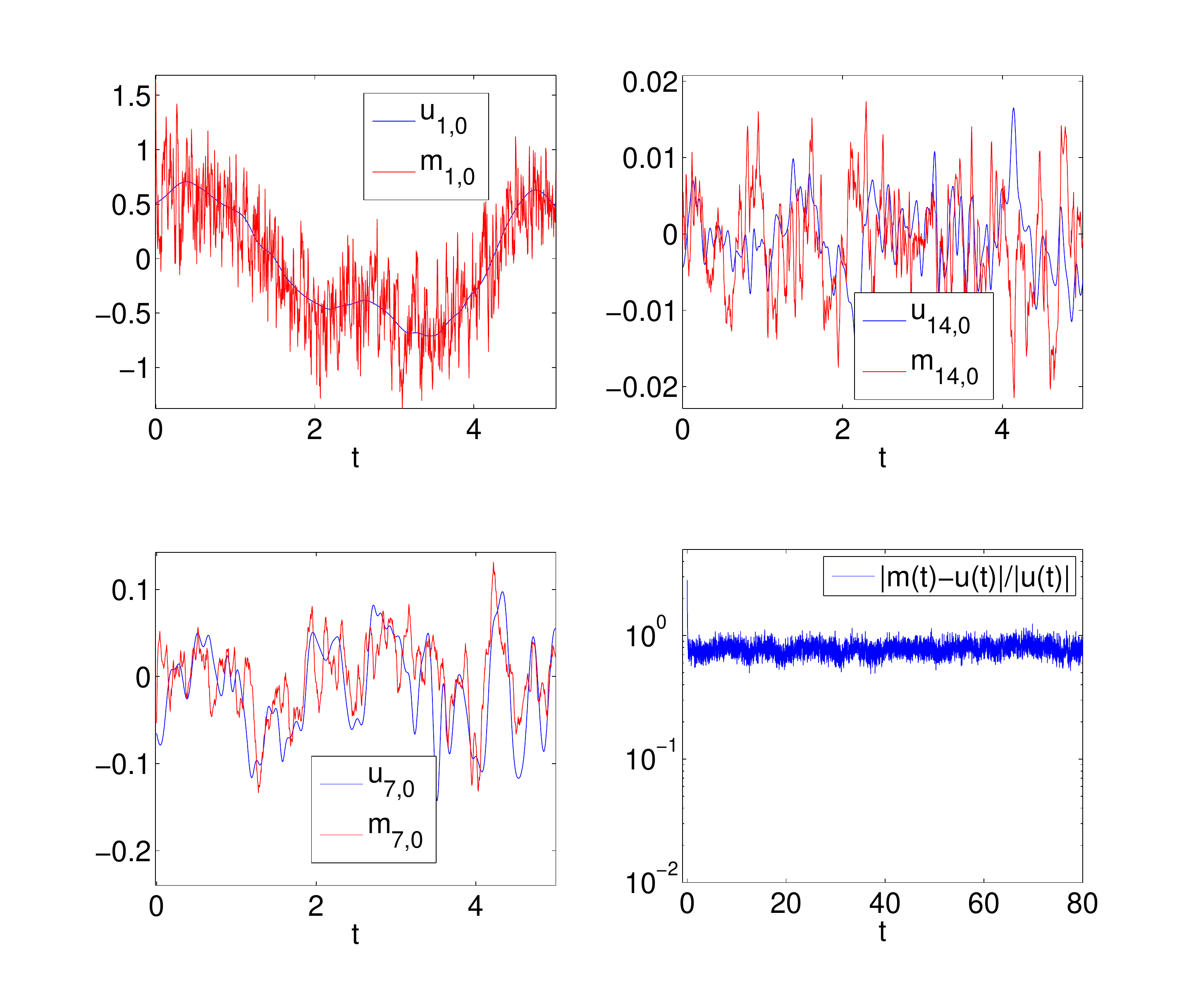}
\caption{Trajectories of various modes of the estimator $\hat{m}$ 
and the signal $u$ are depicted above for 
$\beta=0$ and $\sigma_0=0.05$, along with the total relative 
error in the $L^2$ norm, $|\hat{m}-u|/|u|$.}
\label{a1c5.t}
\end{figure*}

% \begin{figure*}
% \includegraphics[width=1\textwidth]{undamped_noiseB_error.pdf}
% \caption{The relative errors of various modes, given by 
% $|\hat{m}_{k}-u_k|/{\rm max}[0.01,|u_k|]$, and the total relative error
% for $\beta=0$ and $\sigma_0=0.05$.}
% \label{a1c5.e}
% \end{figure*}

\begin{figure*} 
\includegraphics[width=1\textwidth, height=8.5cm]{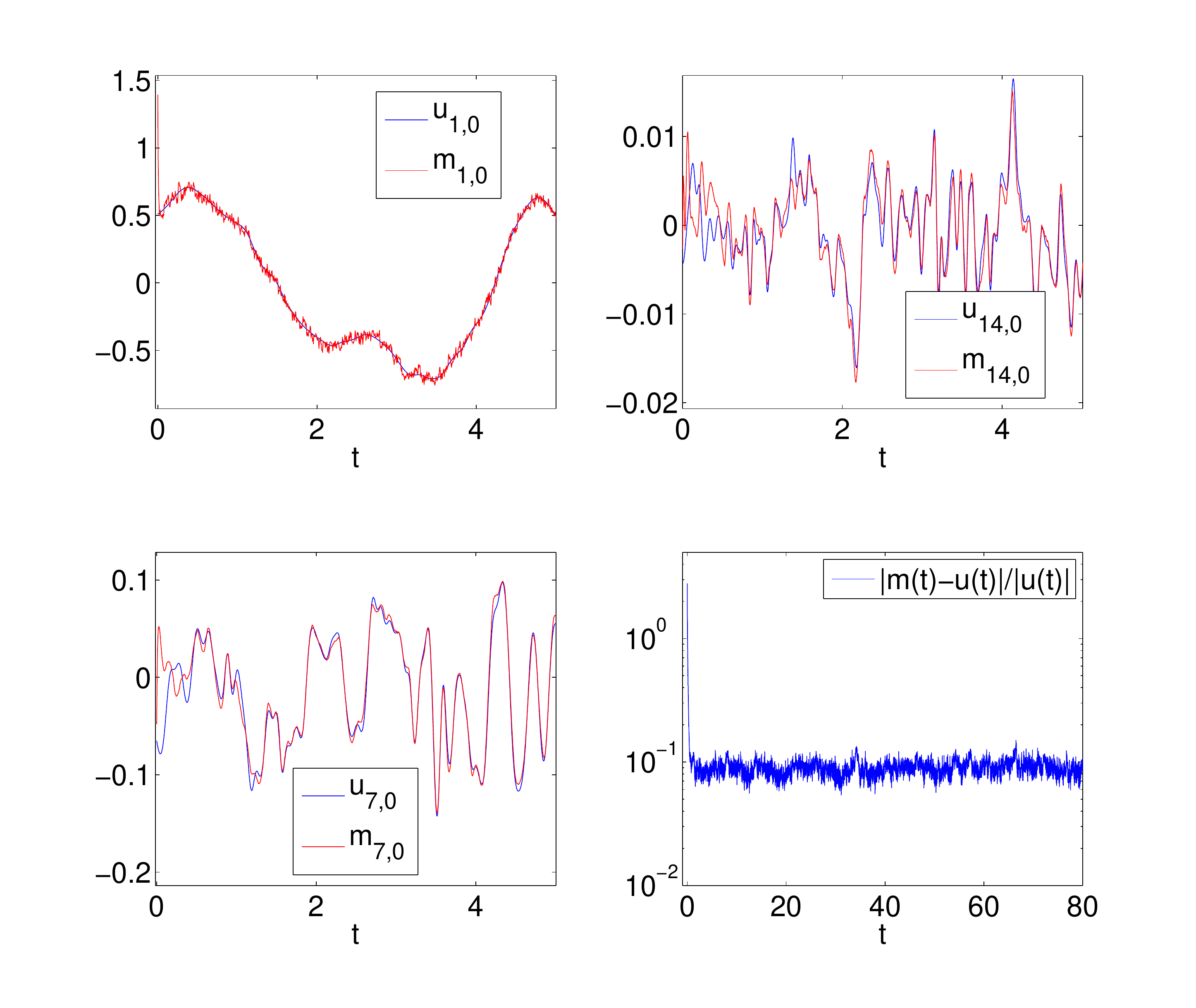}
\caption{Trajectories of various modes of the estimator $\hat{m}$ 
and the signal $u$ are depicted above for $\beta=0$ and 
$\sigma_0=0.005$, along with the relative error
in the $L^2$ norm, $|\hat{m}-u|/|u|$.}
\label{a1c05.t}
\end{figure*}

% \begin{figure*}
% \includegraphics[width=1\textwidth]{undamped_noise_error.pdf}
% \caption{The relative errors of various modes, given by 
% $|\hat{m}_{k}-u_k|/{\rm max}[0.01,|u_k|]$, and the total relative error
% for $\beta=0$ and $\sigma_0=0.005$.}
% \label{a1c05.e}
% \end{figure*}

\begin{figure*}
\includegraphics[width=1\textwidth, height=8.5cm]{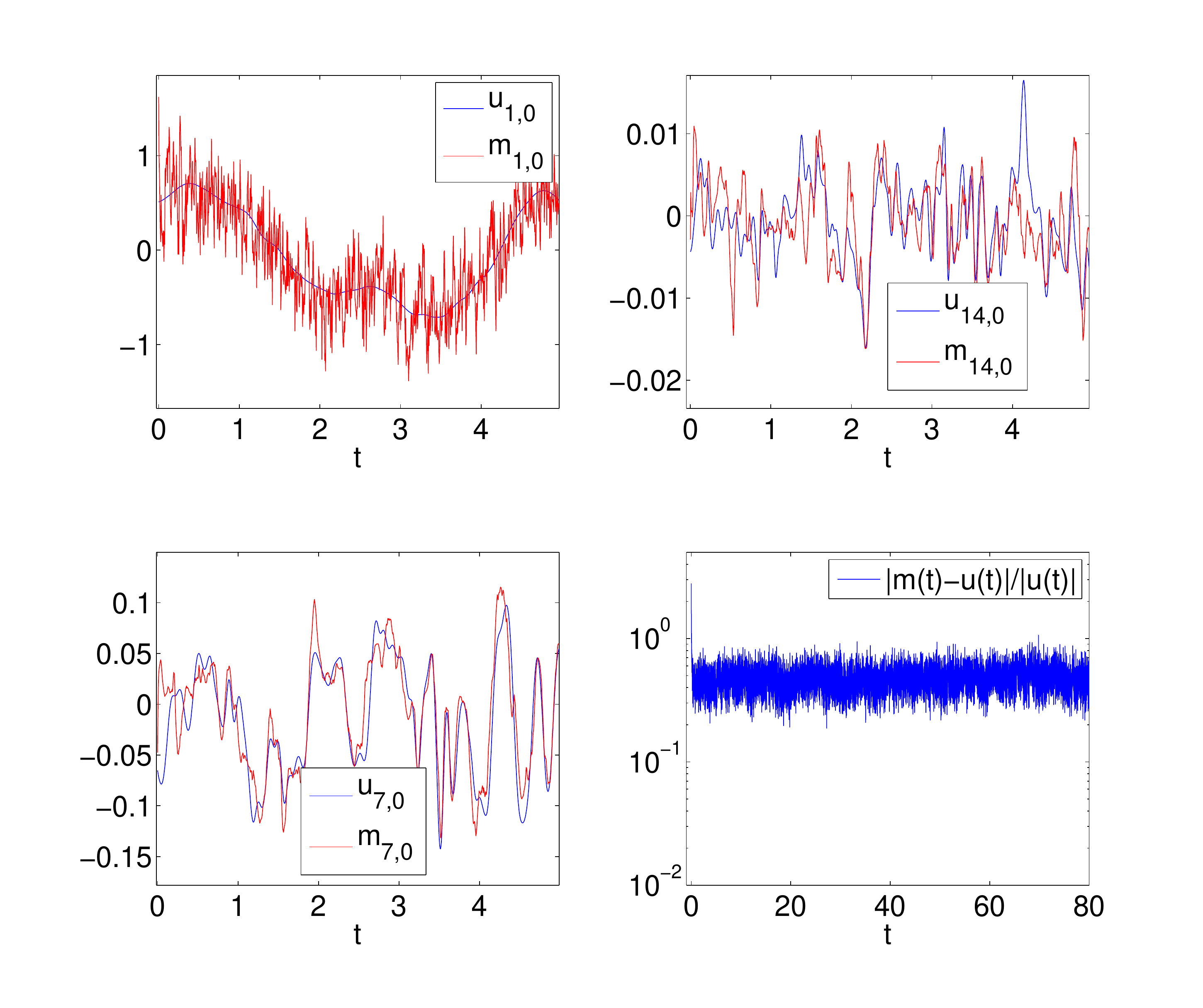}
\caption{Trajectories of various modes of the estimator $\hat{m}$ 
and the signal $u$ are depicted above for 
$\beta=1$ and $\sigma_0=0.05$, along with the relative error
in the $L^2$ norm, $|\hat{m}-u|/|u|$.}
\label{a1dc5.t}
\end{figure*}

% \begin{figure*}
% \includegraphics[width=1\textwidth]{damped_noiseB_error.pdf}
% \caption{The relative errors of various modes, given by 
% $|\hat{m}_{k}-u_k|/{\rm max}[0.01,|u_k|]$, and the total relative error
% for $\beta=2$ and $\sigma_0=0.05$.}
% \label{a1dc5.e}
% \end{figure*}

\begin{figure*} 
\includegraphics[width=1\textwidth, height=8.5cm]{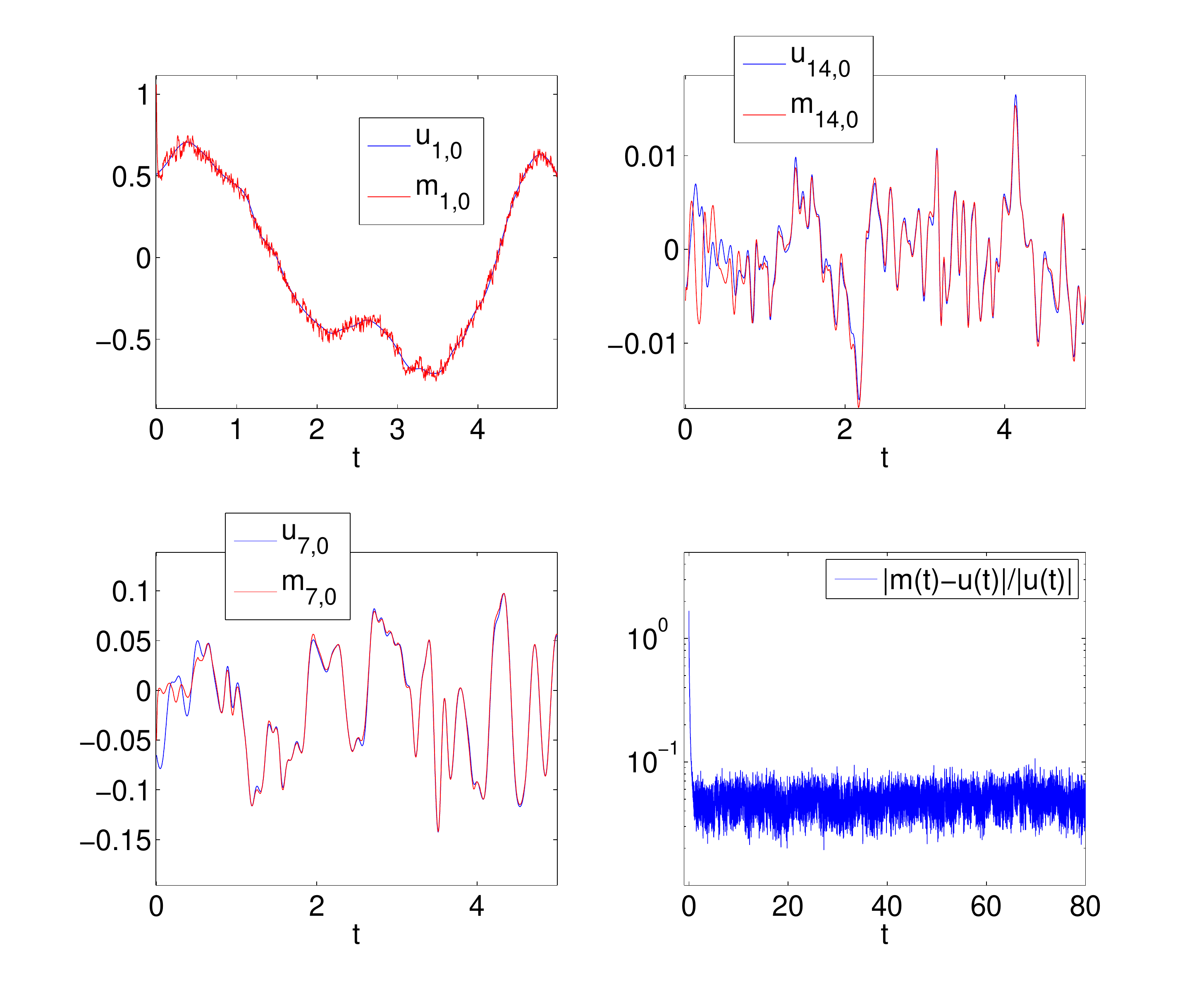}
\caption{Trajectories of various modes of the estimator $\hat{m}$ 
and the signal $u$ are depicted above for $\beta=1$ and 
$\sigma_0=0.005$, along with the relative error
 in the $L^2$ norm, $|\hat{m}-u|/|u|$.}
\label{a1dc05.t}
\end{figure*}

% \begin{figure*}
% \includegraphics[width=1\textwidth]{damped_noise_error.pdf}
% \caption{The relative errors of various modes, given by 
% $|\hat{m}_{k}-u_k|/{\rm max}[0.01,|u_k|]$, and the total relative error
% for $\beta=2$ and $\sigma_0=0.005$.}
% \label{a1dc05.e}
% \end{figure*}

\begin{figure*}[h]
\includegraphics[width=1\textwidth, height=8.5cm]{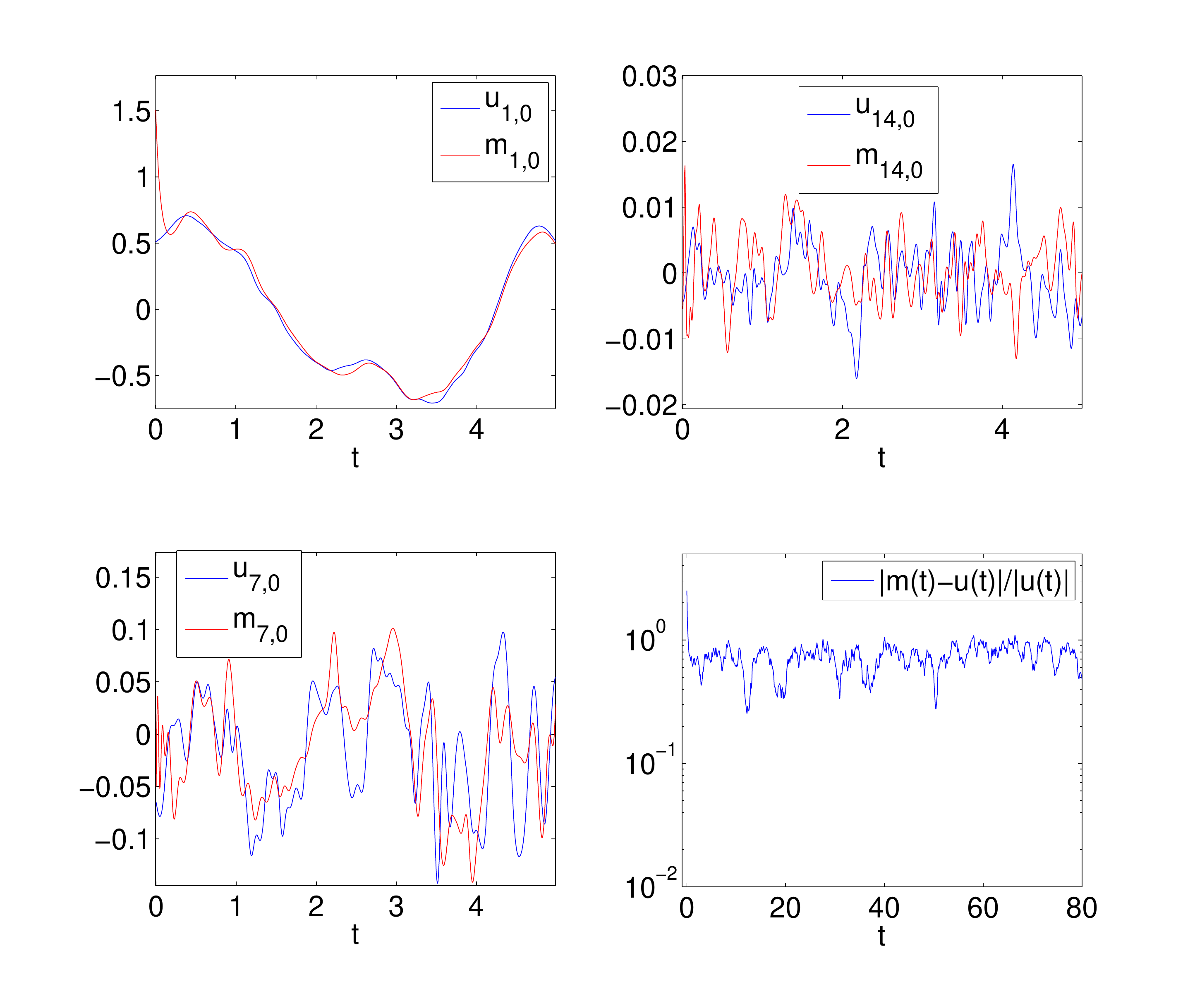}
\caption{Trajectories of various modes of the estimator $\hat{m}$ 
and the signal $u$ are depicted above for $\sigma_{0}=0$ 
and $\omega=10$, along with the relative error
in the $L^2$ norm, $|\hat{m}-u|/|u|$.}
\label{d10.t}
\end{figure*}

% \begin{figure*}
% \includegraphics[width=1\textwidth]{determ10_error.pdf}
% \caption{The relative errors of various modes, given by 
% $|\hat{m}_{k}-u_k|/{\rm max}[0.01,|u_k|]$, and the total relative error
% for $r<1$ and $\omega=10$.}
% \label{d10.e}
% \end{figure*}

\begin{figure*}[h]
\includegraphics[width=1\textwidth, height=8.5cm]{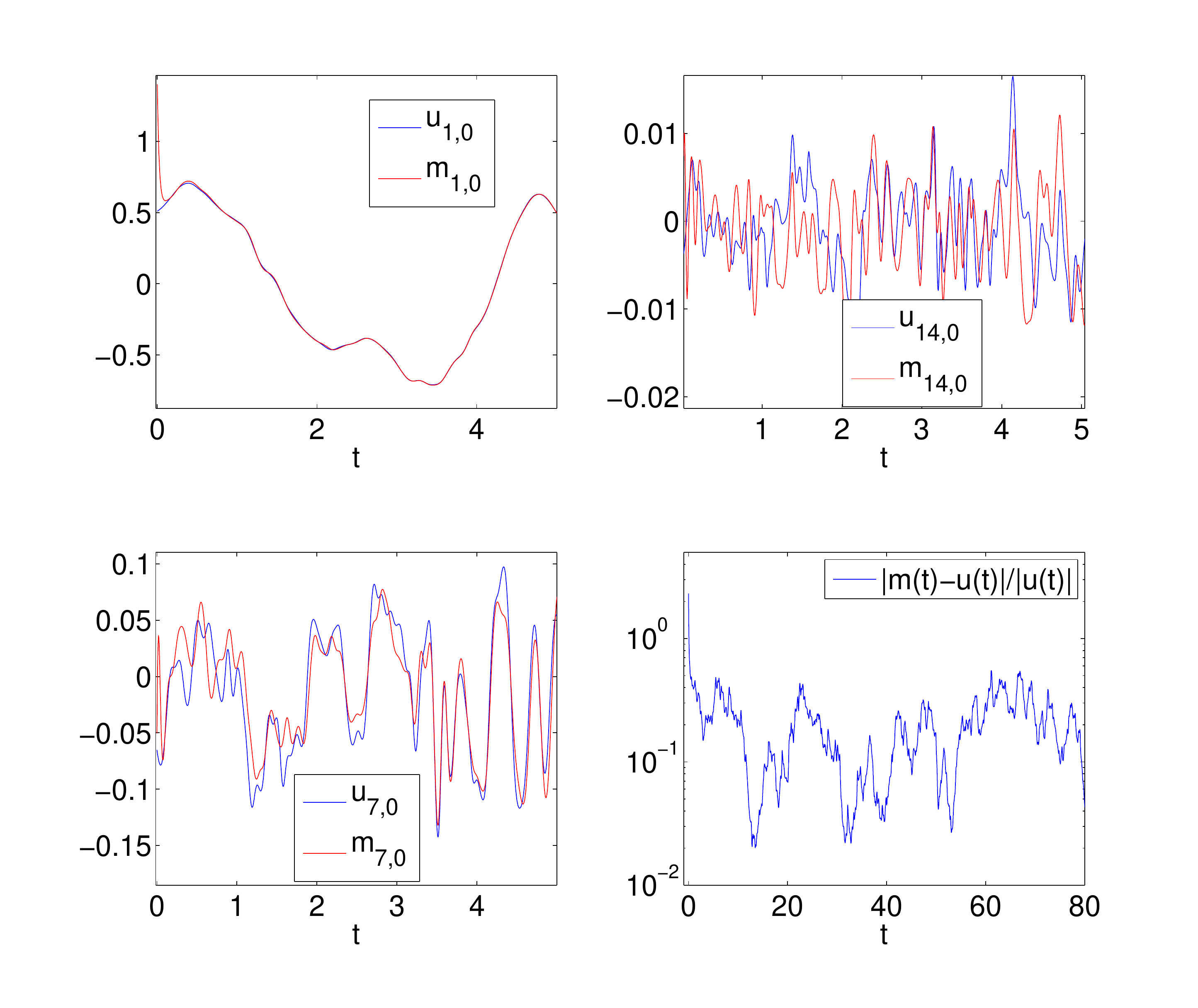}
\caption{Trajectories of various modes of the estimator $\hat{m}$ 
and the signal $u$ are depicted above for
$\sigma_{0}=0$  and $\omega=30$, along with the relative error
in the $L^2$ norm, $|\hat{m}-u|/|u|$.}
\label{d30.t}
\end{figure*}

\subsection{Forward Stability}
\label{sec:num_pullback}

This section will provide numerical evidence supporting
Theorem \ref{t:stabf}.  In order to investigate the 
stability of estimators
%existence and
%properties of (random) attractors, 
we reproduce ensembles of solutions of equation \eqref{eq:nse2}, 
for a fixed realization of $W(t)$, and a family of
initial conditions. 
We let $\beta=0$ throughout this section, and we always
choose values of $\alpha$ which ensure that the trace class
condition on the noise, $4\alpha+2\beta>1$, is satisfied.

Let $m^{(k)}(t)$ be the solution at time $t$
of \eqref{eq:nse2} where the initial conditions
are drawn from a Gaussian whose covariance
is proportional to the model covariance: 
$m^{(k)}(0) \sim \cN(0,30^2 {\hat {C}})$.  
First we consider $\alpha=1/2$.
Figure
\ref{att05} corresponds to parameters given in Fig. \ref{a1c5.t} of 
section \ref{sec:num_forward}.  
The top figure simply shows the ensemble of trajectories, while the bottom 
figure shows the convergence of $|m^{(k)}(t)-m^{(1)}(t)|/|m^{(1)}(t)|$
for $k>1$.  Notice the trajectories converge to each other,
indicating stability. 
%because the attractor for the given
%realization of noise is a single point. 
But, the trajectories here
do not converge
to the truth (or driving signal). 
This is because the neighbourhood of the signal
which bounds the estimators is not small.
%since the attractor for any 
%given realization of the noise is not close to the truth, 
%i.e. the 
%random attractor has a 
%large variance around the truth.  
The next image, Fig. \ref{att005}, 
shows results for the smaller value of $\sigma_0=0.005$ 
corresponding to Fig. \ref{a1c05.t} of section \ref{sec:num_forward}.
Notice the rate of convergence of the trajectories to each other
(bottom) is very similar to the previous case, indicating that there 
is again stability.
%a point attractor for a given realization of noise.
However, this time the neighbourhood of the signal
which bounds the estimators is small, and so they are indeed accurate.
%random attractor has a smaller variance than the previous one, i.e. 
%the trajectories also converge more closely to the truth.
Fig. \ref{noat005} shows the results for the larger value of 
$\alpha=1$ (still with $\beta=0$).  In this case, there is no
stability, %point attractor, 
i.e. the trajectories do not converge to
each other (bottom), and also no convergence to the truth
(bottom right of the top panels), although all trajectories do 
remain in a neighbourhood of the truth and the low wavevector
modes converge (top left), 
so there is accuracy with a large bound.  
Furthermore, the distance of the
trajectories from each other is similar to the distance from the
truth, so the attractor in this case may be similar to the attractor
of the underlying Navier-Stokes equation.

%\note{KCZ: Kody, I find the notation a bit confusing, $m^{(k)},m^{1}$ are 2 different realization of the filter right? Why not call them $m_{1}$ and $m_{2}$ as in Theorem 4.5? Am I missing something here with respect to what $k$ represents?}

\begin{figure*}
\includegraphics[width=1\textwidth, height=8.5cm]{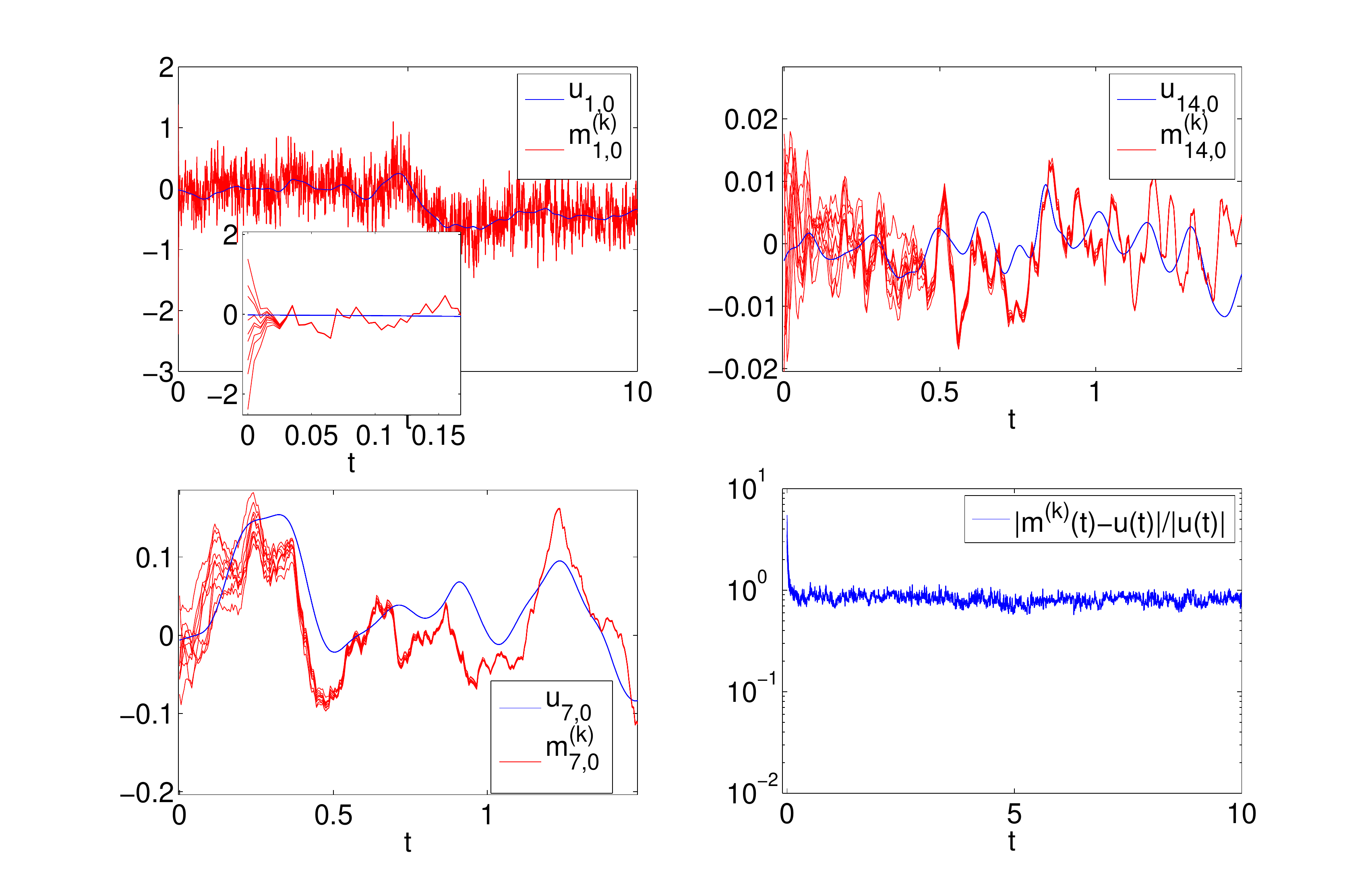}
\includegraphics[width=1\textwidth]{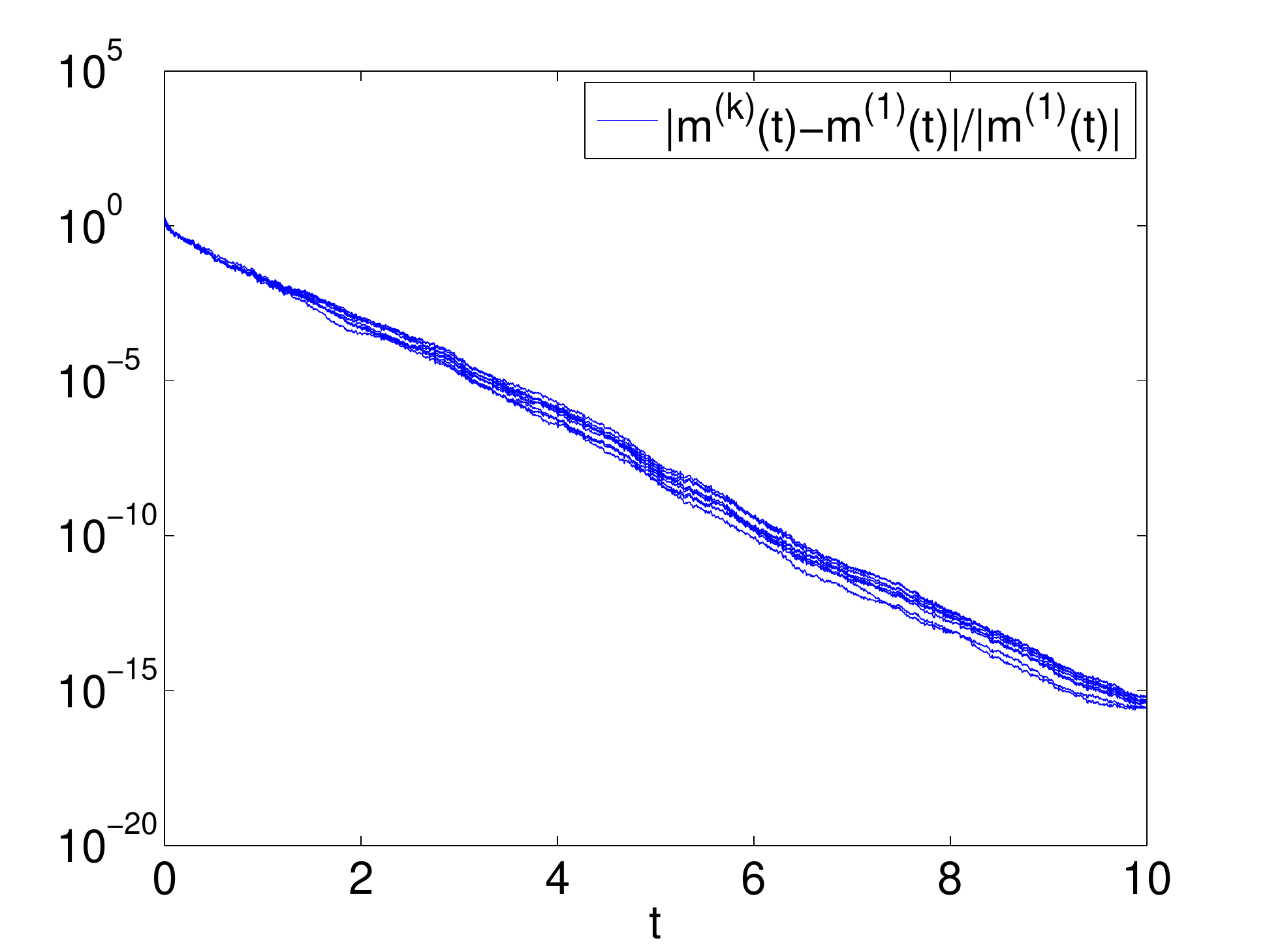}
\caption{The above panels correspond to Fig. \ref{a1c5.t} from the
  text, except illustrating stability by an ensemble of estimators.
The top set of panels are the same as in Fig. \ref{a1c5.t},
while the bottom panel shows stability by convergence of the
estimators to each other.}
\label{att05}
\end{figure*}

\begin{figure*}
\includegraphics[width=1\textwidth, height=8.5cm]{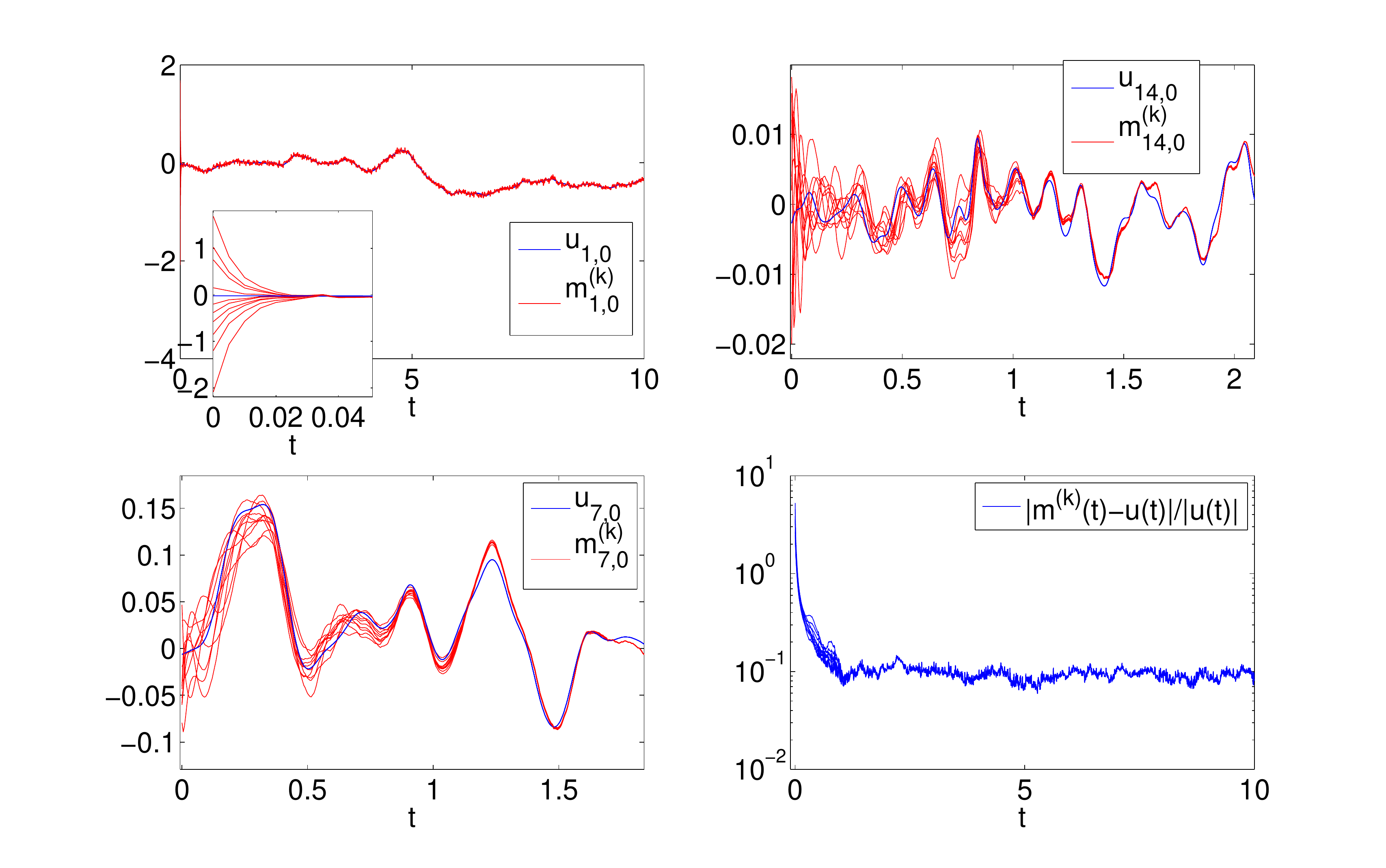}
\includegraphics[width=1\textwidth]{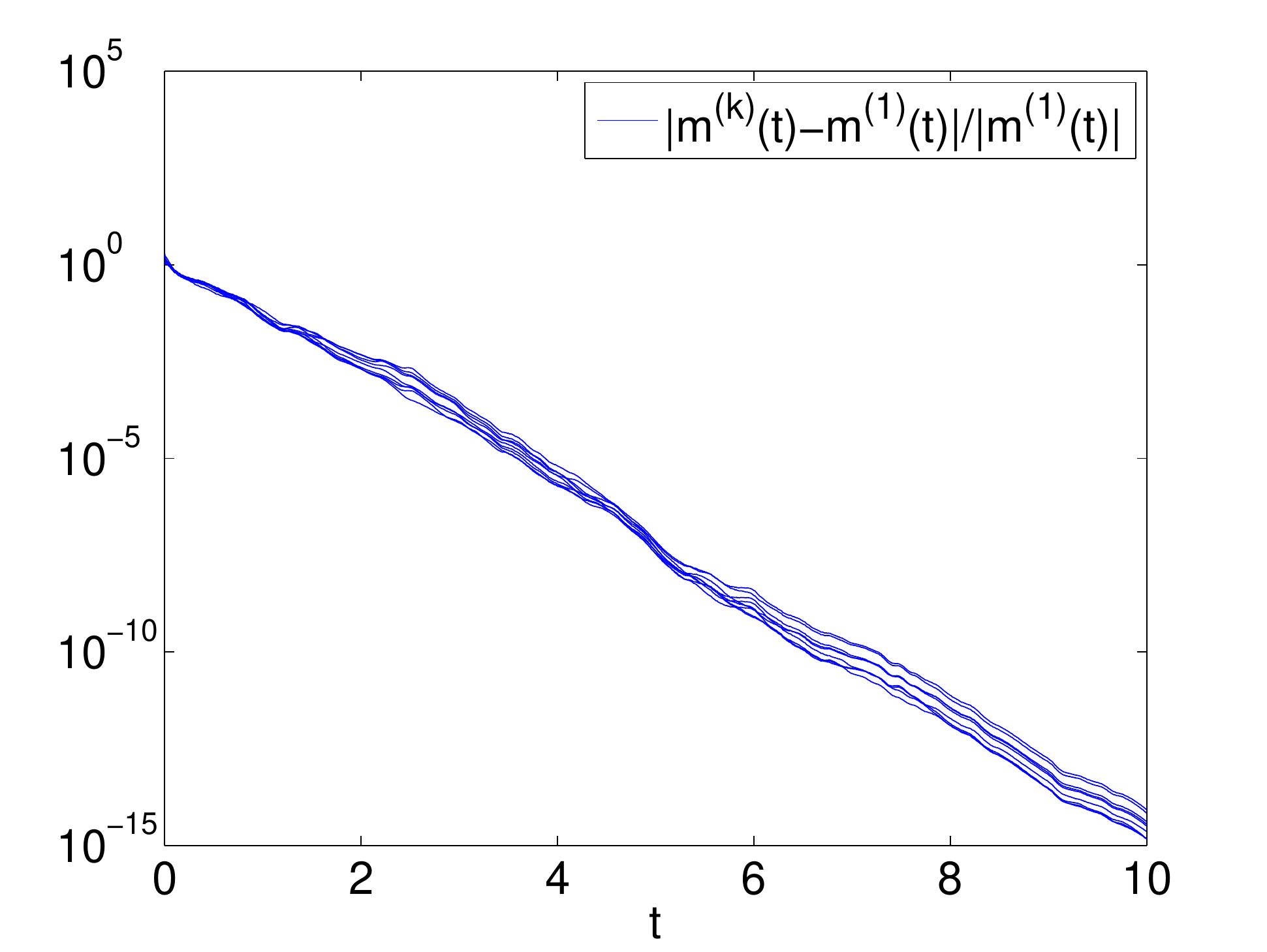}
\caption{The above panels correspond to Fig. \ref{a1c05.t} from the
  text, except illustrating stability by an ensemble of estimators.
The top set of panels are the same as in Fig. \ref{a1c05.t},
while the bottom panel shows stability by convergence of the
estimators to each other.}
\label{att005}
\end{figure*}

\begin{figure*}
\includegraphics[width=1\textwidth,height=8.5cm]{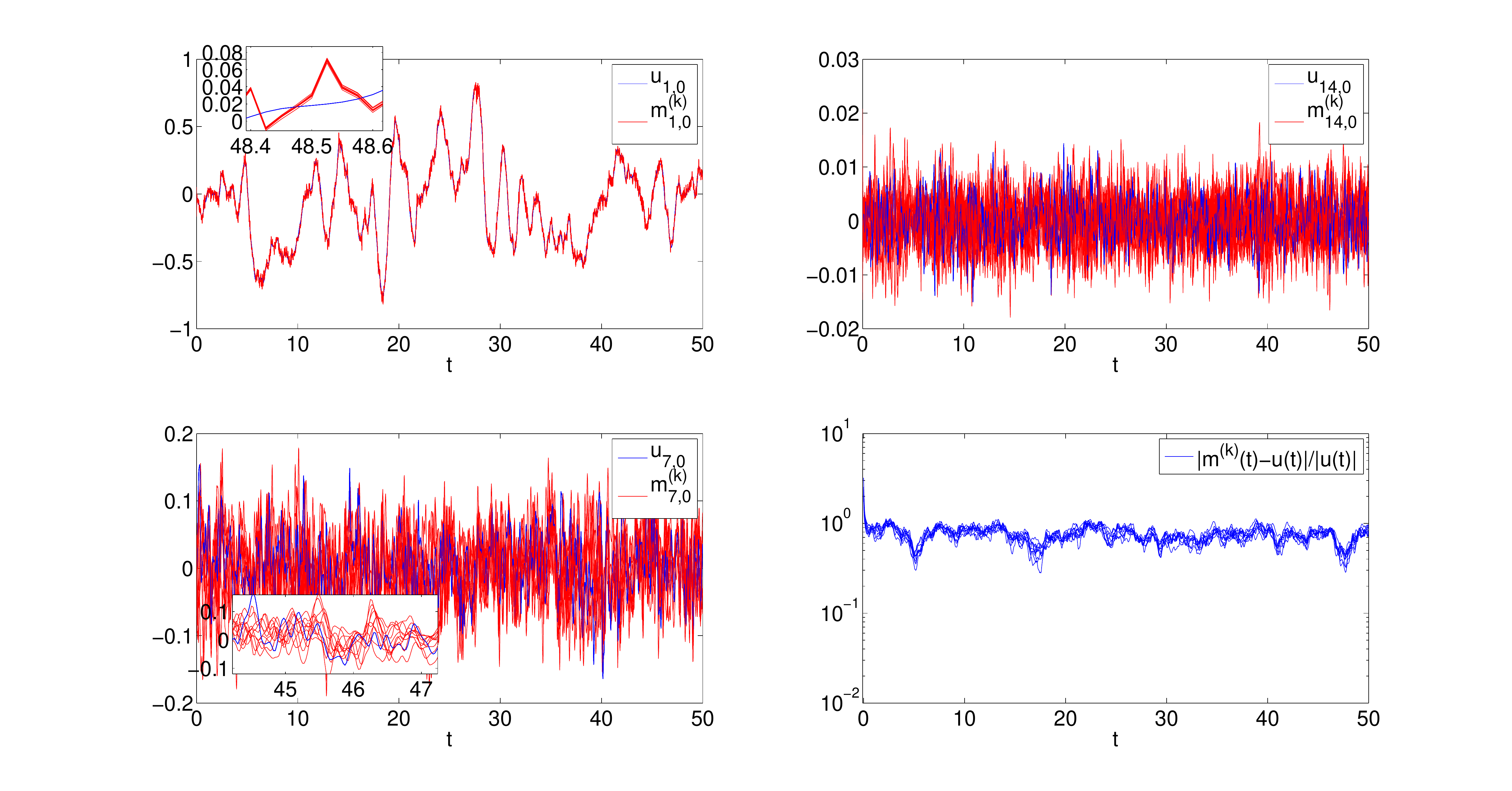}
\includegraphics[width=1\textwidth]{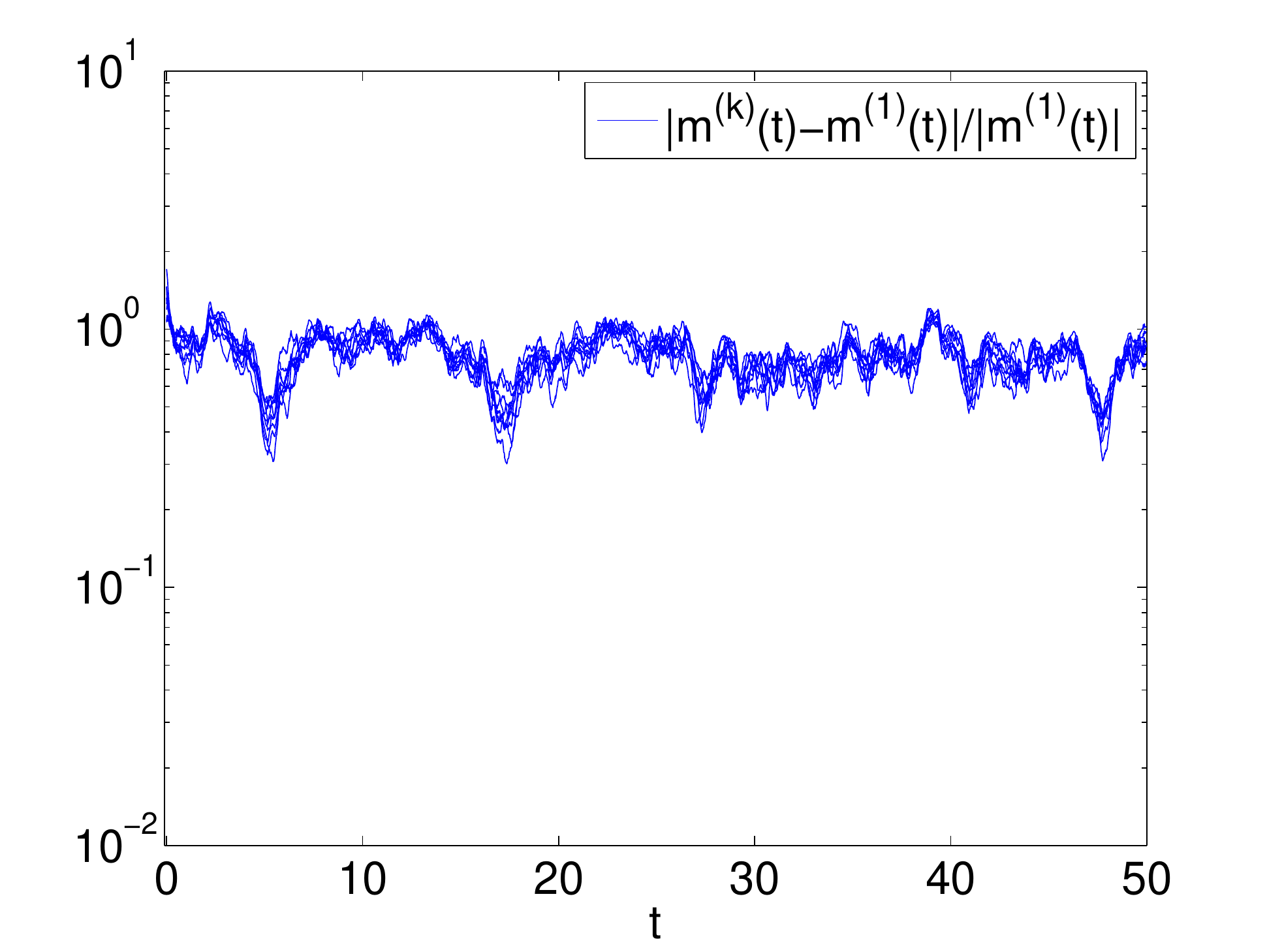}
\caption{The above panels correspond to the same parameter values 
as above Fig. \ref{att005}, except $\alpha=1$. Panels are the same.
There is not stability in this case.}  
%The attractor is definitely not a point
%in this case, as seen in the bottom panel.}
\label{noat005}
\end{figure*}

\subsection{Pullback Accuracy and Stability}

Finally, in this section, we illustrate Theorem \ref{thm:pbacc}.
As the subtle nuance differences between forward and pullback
accuracy and stability ellude standard numerical simulation,
we do not feel it is appropriate to explore this in further detail 
numerically.  So, this section will be brief. 
We include a single image illustrating the equivalence of the above 
experiments in Figures \ref{att005}, \ref{att05}, and \ref{noat005}
to the traditional notion of pullback attractor 
in the case that the attractor is a point: Figure \ref{pbat005}.

\begin{figure*}
\includegraphics[width=1\textwidth, height=8.5cm]{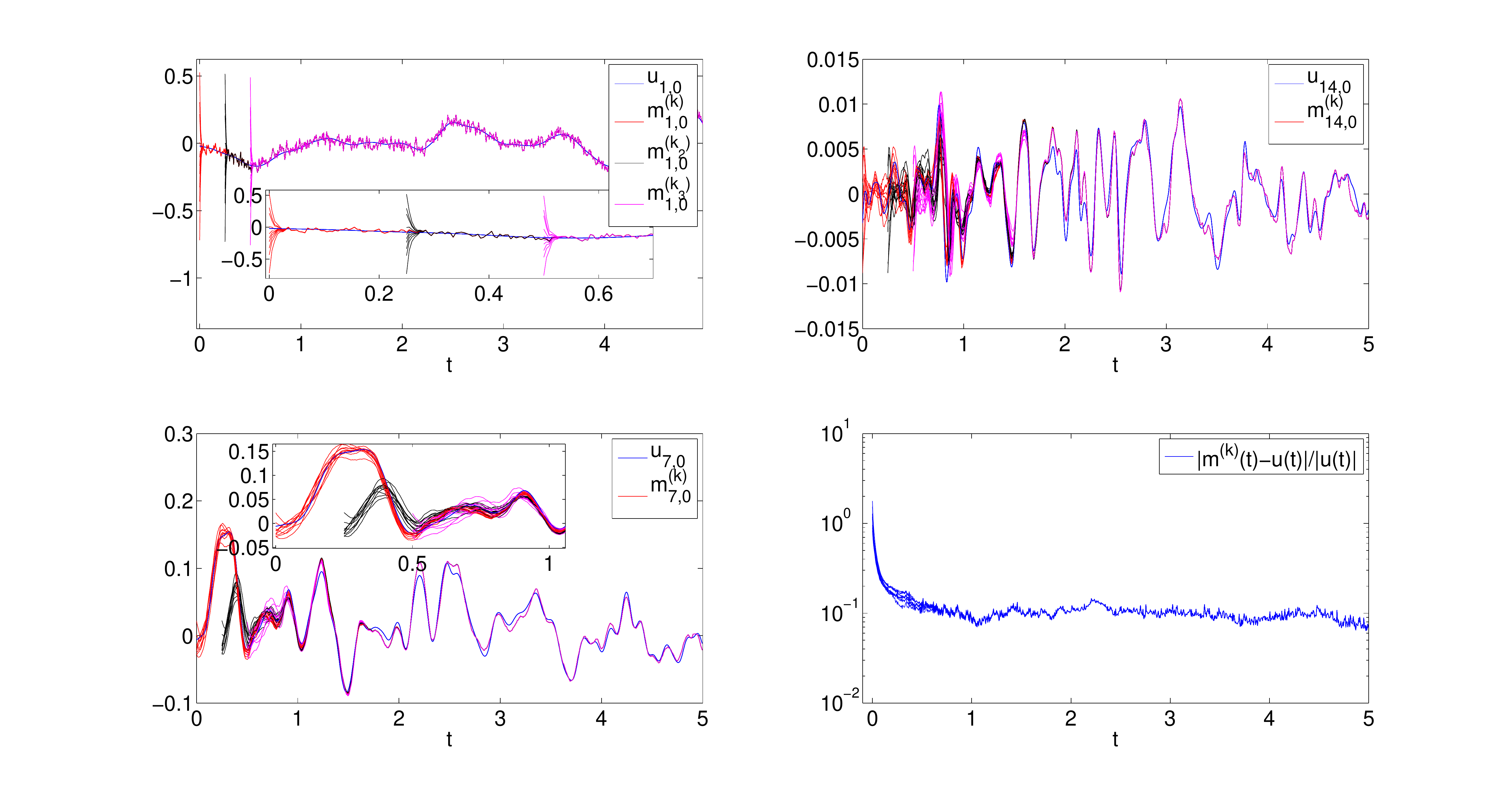}
\caption{The same as Figure \ref{att005}, except the initial ensemble 
is initiated at 3 separate times: $t_1, t_2,$ and $t_3$.  Clearly
the only relevant interval of time is for $t>t_3$.  All trajectories 
converge to each other.}
\label{pbat005}
\end{figure*}

\section{Conclusions}

Data assimilation is important in a range of physical
applications where it is of interest to use data to
improve output from computational models. Analysis of
the various algorithms used in practice is in its infancy.
The work herein contains analysis of an algorithm,
3DVAR, which is prototypical of more complex Gaussian
approximations that are widely used in applications.
In particular we have studied the high frequency in time
observation limit of 3DVAR, leading to a stochastic PDE. 
We have demonstrated mathematically how variance inflation, 
widely used by practitioners, stabilizes, and makes accurate,
this filter, complementing the theory in \cite{lsetal} which
concerns low frequency in time observations. 
It is to be expected that the analytical tools developed
here and in \cite{lsetal} can be built upon to study
more complex algorithms, such as the extended and ensemble
Kalman filters, variants on which are
used in operational weather forecasting.
This will form a focus of our future work.

\bibliographystyle{plain}
\bibliography{fsbib}

\end{document}